\documentclass[11pt]{amsart}
\usepackage{amsopn}
\usepackage{amssymb, amscd}
\usepackage{multirow}
\usepackage{graphicx, graphics, epsfig}
\usepackage{faktor}

\newcommand{\nc}{\newcommand}

\nc{\fg}{\mathfrak{f} } \nc{\vg}{\mathfrak{v} } \nc{\wg}{\mathfrak{w} }
\nc{\zg}{\mathfrak{z} } \nc{\ngo}{\mathfrak{n} } \nc{\kg}{\mathfrak{k} }
\nc{\mg}{\mathfrak{m} } \nc{\bg}{\mathfrak{b} } \nc{\ggo}{\mathfrak{g} }
\nc{\ggob}{\overline{\mathfrak{g}} } \nc{\sog}{\mathfrak{so} }
\nc{\sug}{\mathfrak{su} } \nc{\spg}{\mathfrak{sp} } \nc{\slg}{\mathfrak{sl} }
\nc{\glg}{\mathfrak{gl} } \nc{\cg}{\mathfrak{c} } \nc{\rg}{\mathfrak{r} }
\nc{\hg}{\mathfrak{h} } \nc{\tg}{\mathfrak{t} } \nc{\ug}{\mathfrak{u} }
\nc{\dg}{\mathfrak{d} } \nc{\ag}{\mathfrak{a} } \nc{\pg}{\mathfrak{p} }
\nc{\sg}{\mathfrak{s} } \nc{\affg}{\mathfrak{aff} } \nc{\qg}{\mathfrak{q} }

\nc{\pca}{\mathcal{P}} \nc{\nca}{\mathcal{N}} \nc{\lca}{\mathcal{L}}
\nc{\oca}{\mathcal{O}} \nc{\mca}{\mathcal{M}} \nc{\tca}{\mathcal{T}}
\nc{\aca}{\mathcal{A}} \nc{\cca}{\mathcal{C}} \nc{\gca}{\mathcal{G}}
\nc{\sca}{\mathcal{S}} \nc{\hca}{\mathcal{H}} \nc{\bca}{\mathcal{B}}
\nc{\dca}{\mathcal{D}} \nc{\val}{\operatorname{val}}

\nc{\vp}{\varphi} \nc{\ddt}{\frac{d}{dt}} \nc{\dds}{\frac{d}{ds}}
\nc{\dpar}{\frac{\partial}{\partial t}} \nc{\im}{\mathrm{i}}

\nc{\SO}{\mathrm{SO}} \nc{\Spe}{\mathrm{Sp}} \nc{\Sl}{\mathrm{SL}}
\nc{\SU}{\mathrm{SU}} \nc{\Or}{\mathrm{O}} \nc{\U}{\mathrm{U}} \nc{\Gl}{\mathrm{GL}}
\nc{\Se}{\mathrm{S}} \nc{\Cl}{\mathrm{Cl}} \nc{\Spein}{\mathrm{Spin}}
\nc{\Pin}{\mathrm{Pin}} \nc{\G}{\mathrm{GL}_n(\RR)} \nc{\g}{\mathfrak{gl}_n(\RR)}

\nc{\RR}{{\Bbb R}} \nc{\HH}{{\Bbb H}} \nc{\CC}{{\Bbb C}} \nc{\ZZ}{{\Bbb Z}}
\nc{\FF}{{\Bbb F}} \nc{\NN}{{\Bbb N}} \nc{\QQ}{{\Bbb Q}} \nc{\PP}{{\Bbb P}} \nc{\OO}{{\Bbb O}}

\nc{\vs}{\vspace{.2cm}} \nc{\vsp}{\vspace{1cm}} \nc{\ip}{\langle\cdot,\cdot\rangle}
\nc{\ipp}{(\cdot,\cdot)} \nc{\la}{\langle} \nc{\ra}{\rangle} \nc{\unm}{\tfrac{1}{2}}
\nc{\unc}{\tfrac{1}{4}} \nc{\und}{\tfrac{1}{16}} \nc{\no}{\vs\noindent}
\nc{\lam}{\Lambda^2(\RR^n)^*\otimes\RR^n} \nc{\tangz}{{\rm T}^{\rm Zar}}
\nc{\nor}{{\sf n}}  \nc{\mum}{/\!\!/} \nc{\kir}{/\!\!/\!\!/}
\nc{\Ri}{\tfrac{4\Ric_{\mu}}{||\mu||^2}} \nc{\ds}{\displaystyle}
\nc{\ben}{\begin{enumerate}} \nc{\een}{\end{enumerate}} \nc{\f}{\frac}
\nc{\lb}{[\cdot,\cdot]} \nc{\isn}{\tfrac{1}{||v||^2}}
\nc{\gkp}{(\ggo=\kg\oplus\pg,\ip)} \nc{\ukh}{(\ug=\kg\oplus\hg,\ip)}
\nc{\tgkp}{(\tilde{\ggo}=\kg\oplus\pg,\ip)}
\nc{\wt}{\widetilde} \nc{\mm}{M}
\nc{\iop}{\mathtt{i}} \nc{\jop}{\mathtt{j}}

\nc{\Hess}{\operatorname{Hess}} \nc{\ad}{\operatorname{ad}}
\nc{\Ad}{\operatorname{Ad}} \nc{\rank}{\operatorname{rank}}
\nc{\Irr}{\operatorname{Irr}} \nc{\End}{\operatorname{End}}
\nc{\Aut}{\operatorname{Aut}} \nc{\Inn}{\operatorname{Inn}}
\nc{\Der}{\operatorname{Der}} \nc{\Ker}{\operatorname{Ker}}
\nc{\Iso}{\operatorname{Iso}} \nc{\Diff}{\operatorname{Diff}}
\nc{\Lie}{\operatorname{L}} \nc{\tr}{\operatorname{tr}} \nc{\dif}{\operatorname{d}}
\nc{\sen}{\operatorname{sen}} \nc{\modu}{\operatorname{mod}}
\nc{\CRic}{\operatorname{PP}} \nc{\Cric}{\operatorname{P}} \nc{\Ricci}{\operatorname{Ric}} \nc{\Sec}{\operatorname{Sec}}
\nc{\sym}{\operatorname{sym}} \nc{\herm}{\operatorname{herm}} \nc{\symac}{\operatorname{sym^{ac}}}
\nc{\symc}{\operatorname{sym^{c}}} \nc{\scalar}{\operatorname{sc}}
\nc{\grad}{\operatorname{grad}} \nc{\ricci}{\operatorname{Rc}}
\nc{\Nor}{\operatorname{Norm}}  \nc{\ricc}{\operatorname{Rc^{c}}}
\nc{\Ricc}{\operatorname{Ric^{c}}} \nc{\ricac}{\operatorname{Rc^{ac}}}
\nc{\Ricac}{\operatorname{Ric^{ac}}} \nc{\Riem}{\operatorname{Rm}}
\nc{\riccig}{\operatorname{ric^{\gamma}}} \nc{\Rin}{\operatorname{M}}
\nc{\Le}{\operatorname{L}} \nc{\tang}{\operatorname{T}}
\nc{\level}{\operatorname{level}} \nc{\rad}{\operatorname{r}}
\nc{\abel}{\operatorname{ab}} \nc{\CH}{\operatorname{CH}} \nc{\Cone}{\operatorname{C}} \nc{\CCone}{\operatorname{CC}}
\nc{\mcc}{\operatorname{mcc}} \nc{\Adj}{\operatorname{Adj}}
\nc{\Order}{\operatorname{O}}  \nc{\inj}{\operatorname{inj}} \nc{\proy}{\operatorname{pr}}
\nc{\vol}{\operatorname{vol}} \nc{\Diag}{\operatorname{Dg}} \nc{\Diagg}{\operatorname{Diag}}
\nc{\Spec}{\operatorname{Spec}} \nc{\Ima}{\operatorname{Im}} \nc{\Rea}{\operatorname{Re}}
\nc{\spann}{\operatorname{span}}

\theoremstyle{plain}
\newtheorem{theorem}{Theorem}[section]
\newtheorem{proposition}[theorem]{Proposition}
\newtheorem{corollary}[theorem]{Corollary}
\newtheorem{lemma}[theorem]{Lemma}

\theoremstyle{definition}
\newtheorem{definition}[theorem]{Definition}

\theoremstyle{remark}
\newtheorem{remark}[theorem]{Remark}

\newtheorem{example}[theorem]{Example}

\title{On Ricci negative solvmanifolds and their nilradicals}

\author{Jonas Der\'e} \author{Jorge Lauret}

\address{KU Leuven Kulak, E. Sabbelaan 53, BE-8500 Kortrijk, Belgium}
\email{jonas.dere@kuleuven-kulak.be}

\address{Universidad Nacional de C\'ordoba, FaMAF and CIEM, 5000 C\'ordoba, Argentina}
\email{lauret@famaf.unc.edu.ar}

\thanks{The first author is supported by a postdoctoral fellowship of the Research Foundation - Flanders (FWO).  The second author was partially supported by grants from CONICET, FONCYT and SeCyT (Universidad Nacional de C\'ordoba). }

\begin{document}

\maketitle

\begin{abstract}
In the homogeneous case, the only curvature behavior which is still far from being understood is Ricci negative.  In this paper, we study which nilpotent Lie algebras admit a Ricci negative solvable extension.  Different unexpected behaviors were found.  On the other hand, given a nilpotent Lie algebra, we consider the space of all the derivations such that the corresponding solvable extension has a metric with negative Ricci curvature.  Using the nice convexity properties of the moment map for the variety of nilpotent Lie algebras, we obtain a useful characterization of such derivations and some applications.
\end{abstract}

\tableofcontents

\section{Introduction}\label{intro}

There are no topological obstructions on a differentiable manifold $M$ to the existence of a complete Riemannian metric with negative Ricci curvature (see \cite{Lhk}).  However, in the presence of a Lie group $G$ acting transitively on $M$, it is natural to expect a nice interplay between any prescribed curvature behavior of $G$-invariant metrics and not only the topology of $M$ but also the algebraic structure of $G$.

Back in 1974, Heintze \cite{Hnt} (see also \cite{AznWls}) proved that any homogeneous Riemannian manifold with $\Sec<0$ is isometric to a metric on a simply connected solvable Lie group (any metric on a Lie group is assumed to be left-invariant from now on) satisfying the following strong structural property: the nilradical $\ngo$ of its Lie algebra $\sg$ has codimension one and there is an element $Y\in\sg$ such that the derivation $\ad{Y}|_{\ngo}$ of $\ngo$ is {\it positive}, in the sense that its real semisimple part $\ad{Y}|_{\ngo}^\RR$, which is also in $\Der(\ngo)$, has its eigenvalues (i.e.\ the real parts of the eigenvalues of $\ad{Y}|_{\ngo}$) all positive.  Conversely, any solvable Lie group of this kind admits a metric with $\Sec<0$.  Surprisingly (or not), the obvious question of which nilpotent Lie algebras admit a positive derivation is still wide open.  We will show in this paper that such a problem seems to be hopeless.

The stronger pinching condition $-4\leq\Sec\leq-1$ was studied by Eberlein-Heber \cite{EbrHbr}; for instance, they showed that one additionally needs $\ngo$ $2$-step nilpotent (or abelian) and $\Spec\left(\ad{Y}|_{\ngo}^\RR\right)\subset [1,2]$.  On the other hand, concerning the weaker condition $\Sec\leq 0$, it was proved by Azencott-Wilson \cite{AznWls} (see also \cite{Wlf,Alk}) that the only homogeneous examples (up to isometry) are still simply connected solvable Lie groups.  Here the orthogonal complement $\ag$ of the nilradical $\ngo$ in $\sg$ can be of dimension $>1$, though the conditions $[\ag,\ag]=0$ and $\ad{Y}|_{\ngo}^\RR\geq 0$ must hold, among other more technical conditions.

In the homogeneous case, the only curvature behavior which is still not understood is $\Ricci<0$ (see e.g.\ \cite[Introduction]{NklNkn}).  In the 1980s, Dotti-Leite-Miatello \cite{Dtt,DttLt, DttLtMtl} proved that the only unimodular Lie groups that can admit a $\Ricci<0$ metric are the non-compact semisimple ones and showed that most of non-compact simple Lie groups indeed have one, with some low dimensional exceptions, including $\Sl_2(\CC)$, $\Spe(2,\RR)$ and $G_2$ (non-compact).  The existence of $\Ricci<0$ metrics on such groups is still open, the only solved case is $\Sl_2(\RR)$, where the non-existence easily follows (see e.g.\ \cite{Mln}).  It was proved by Jablonski-Petersen \cite{JblPtr} that a semisimple Lie group admitting a metric with $\Ricci<0$ can not have compact factors, i.e.\ it is of non-compact type.  Recall that topologically, any (connected) Lie group is a product $K\times\RR^m$, where $K$ is its maximal compact subgroup.

More recently, in 2016, unexpected examples of Lie groups admitting $\Ricci<0$ metrics which are neither semisimple nor solvable were constructed by Will \cite{Wll1, Wll2}.  The Levi factors of some of these examples are compact, including $\SU(n)$ ($n\geq 2$) and $\SO(n)$ ($n\geq 3$), and therefore four of the nine topologies missed by the semisimple examples in \cite{DttLtMtl} are attained: $K\times\RR^m$ for $K$ equal to $\SU(2)$, $\SU(3)$, $\SO(5)$ or $\SO(7)$.
The cases in which $K$ is $S^1$, $\Spe(3)$, $\Spe(4)$, $\Spe(5)$ or $G_2$ remain open.  It is worth pointing out that the homogeneous space $\SO^+(n,2)/\SO(n)$ ($n\geq 2$), which is homeomorphic to $S^1\times\RR^k$, does admit an invariant metric with $\Ricci<0$ (see \cite[Example 1]{Nkn}).  On the other hand, a general construction in \cite{Wll2} gives that any non-compact semisimple Lie group admitting a $\Ricci<0$ metric can be the Levi factor of a non-semisimple Lie group with a $\Ricci<0$ metric.  Non-abelian nilradicals are possible in most of these Will's constructions.  All this shows that an algebraic characterization of Lie groups having a $\Ricci<0$ metric is out of reach at the moment.

The study of the solvable case was also recently initiated by Nikolayevsky-Nikonorov \cite{NklNkn} in 2015.  They obtained the following sufficient condition on a solvable Lie group $S$ to admit a metric with $\Ricci<0$:
\begin{equation}\label{NNsuf}
\mbox{There exists $Y\in\sg$ such that $\ad{Y}|_\ngo^\RR>0$,     }
\end{equation}
where $\ngo$ is the nilradical of $\sg$.  Note that the nilradicals involved in these examples are the same as those needed for $\Sec<0$, although the condition $[\ag,\ag]=0$ is not mandatory here as in the case of $\Sec\leq 0$.  Also a necessary condition was found in \cite{NklNkn}:
\begin{equation}\label{NNnec}
\mbox{There exists $Y\in\sg$ such that $\tr{\ad{Y}}>0$ and $\ad{Y}|_{\zg(\ngo)}^\RR>0$,    }
\end{equation}
where $\zg(\ngo)$ is the center of $\ngo$.

We note that all the structural conditions on a solvable Lie group related to the existence of negative sectional or Ricci curvature metrics have the same flavor, motivating the following fundamental question:
\begin{quote}
Which nilpotent Lie algebras can be the nilradical of some solvable Lie algebra admitting a $\Ricci<0$ metric?
\end{quote}

Such a Lie algebra will be called a {\it Ricci negative nilradical} (RN-nilradical for short).  Since the existence of a positive derivation is sufficient, even for $\Sec<0$, any nilpotent Lie algebra which is $2$-step or has dimension $\leq 6$ is a RN-nilradical.

In Section \ref{RNnil}, we first show that for a nilpotent Lie algebra, the existence of a derivation of positive trace is a condition that is stronger than admitting a non-trivial diagonalizable derivation, and that this is in turn stronger than the property of having only nilpotent derivations.  Secondly, we use condition \eqref{NNnec} (the only obstruction known) to exhibit many explicit examples of nilpotent Lie algebras which are not RN-nilradicals.  They all have a derivation of positive trace and the following characteristics are obtained (for the first three examples any diagonalizable derivation has a zero eigenvalue on the center):
\begin{itemize}
  \item $\dim{\ngo}=8$.

  \item $\ngo$ is $3$-step nilpotent.

  \item A continuous family of pairwise non-isomorphic algebras of dimension $13$.

  \item $\ngo$ has a non-singular derivation but any diagonalizable derivation has a negative eigenvalue on the center.
\end{itemize}

On the contrary, we show that the fact that any diagonalizable derivation of $\ngo$ has a negative eigenvalue is not an obstacle for $\ngo$ to be a RN-nilradical.  All this suggests that, as in the study of Einstein nilradicals (see e.g.\ \cite{cruzchica}), the search for new sufficient or necessary general conditions is a challenging problem.

In the light of the results obtained in \cite{NklNkn,Nkl} in the general case as well as in the particular cases of Heisenberg and filiform Lie algebras as nilradicals, a complete characterization of solvable Lie algebras admitting $\Ricci<0$ metrics is expected to take the following form:
\begin{quote}
  There exists $Y\in\sg$ such that $\ad{Y}|_{\ngo}^\RR$ belongs (up to automorphism conjugation) to certain open and convex cone in the maximal torus of derivations of the nilradical $\ngo$ of $\sg$.
\end{quote}

We study this problem in Section \ref{RNder}.  At the core of this question one has the following situation.  Given a nilpotent Lie algebra $\ngo$, each $D\in\Der(\ngo)$ defines a solvable Lie algebra $\sg_D=\RR f\oplus\ngo$ given as the semi-direct sum such that $\ad{f}|_{\ngo}=D$.  We call $D$ {\it Ricci negative} if $\tr{D}>0$ and $\sg_D$ admits a $\Ricci<0$ metric such that $D^t=D$ and $f\perp\ngo$.  Note that any $D>0$ is Ricci negative (see \eqref{NNsuf}) and $D|_{\zg(\ngo)}>0$ is a necessary condition (see \eqref{NNnec}).  The following natural questions arise:
\begin{quote}
Given a basis $\{ e_i\}$ of $\ngo$, what kind of set is the cone of Ricci negative diagonal derivations? Is it open in the space of diagonal derivations?  Is it convex?
\end{quote}

We prove that a diagonal derivation $D$ is Ricci negative if and only if $D$ belongs to certain open and convex cone depending on $D$ (see Corollary \ref{main}).  Our main tool is the moment map for the $\Gl(\ngo)$-action on the variety of nilpotent Lie algebras, which is known from real geometric invariant theory to satisfy nice convexity properties (see \cite{HnzSch}).  In the case when the basis $\{ e_i\}$ is nice (see Definition \ref{nice-def}), a particularly neat characterization of Ricci negative derivations is given.  As an application, we obtain that any nilpotent Lie algebra of dimension $7$ having a non-nilpotent derivation is a RN-nilradical (see Theorem \ref{dim7}).  More applications are developed in the forthcoming paper \cite{RNder}.

\section{Preliminaries}\label{preli}

\subsection{The representation $\lam$}\label{V-sec}
We consider the space of all skew-symmetric algebras of dimension $n$, which is parameterized by the vector space
$$
V:=\lam.
$$
There is a natural linear action of $\G$ on $V$ given by $g\cdot\mu:=g\mu(g^{-1}\cdot,g^{-1}\cdot)$, for all $g\in\G$, $\mu\in V$, whose derivative defines the $\g$-representation on $V$,
$$
E\cdot\mu=E\mu(\cdot,\cdot)-\mu(E\cdot,\cdot)-\mu(\cdot,E\cdot),\qquad E\in\g,\quad\mu\in V.
$$
We note that $E\cdot\mu=0$ if and only if $E\in\Der(\mu)$, the Lie algebra of derivations of the algebra $\mu$.  Let $\tg^n$ denote the set of all diagonal $n\times n$ matrices.  If $\{ e^1,...,e^n\}$ is the basis of $(\RR^n)^*$ dual to the canonical basis $\{ e_1,...,e_n\}$, then
$$
\{ \mu_{ijk}:=(e^i\wedge e^j)\otimes e_k : 1\leq i<j\leq n, \; 1\leq k\leq n\}
$$
is a basis of $V$ of weight vectors for the above representation.  Note that $\mu_{ijk}$ is actually the bilinear form on $\RR^n$ defined by
$\mu_{ijk}(e_i,e_j)=-\mu_{ijk}(e_j,e_i)=e_k$ and zero otherwise.  The corresponding weights are given by
$$
F_{ij}^k:=E_{kk}-E_{ii}-E_{jj}\in\tg^n, \qquad i<j,
$$
where $E_{rs}$ denotes as usual the matrix whose only nonzero coefficient is $1$ at entry $rs$.  The structural constants $c(\mu)_{ij}^k$ of an
algebra $\mu\in V$ are then given by
$$
\mu(e_i,e_j)=\sum_{k}c(\mu)_{ij}^k\, e_k,
\qquad \mu=\sum_{i<j,\,k} c(\mu)_{ij}^k\, \mu_{ijk}.
$$
We consider the Weyl chamber of $\g$ defined by
$$
\tg^n_+:=\left\{\Diag(a_1,\dots,a_n):a_1\leq...\leq a_n\right\},
$$
and the open cone
$$
\tg^n_{>0}:=\left\{\Diag(a_1,\dots,a_n):a_i>0\right\}.
$$
The canonical inner product $\ip$ on $\RR^n$ determines $\Or(n)$-invariant inner products on $V$ and $\g$ making of $\{ \mu_{ijk}\}$ and $\{ E_{ij}\}$ orthonormal bases, respectively.  All these inner products will also be denoted by $\ip$.

\subsection{Moment map}\label{mm-sec}
The moment map (or $\G$-gradient map) from geometric invariant theory (see e.g.\ \cite{HnzSchStt,HnzSch,BhmLfn} for further information) for the above representation is the $\Or(n)$-equivariant map
$$
m:V\smallsetminus\{ 0\}\longrightarrow\sym(n),
$$
defined implicitly by
\begin{equation}\label{defmm}
\la m(\mu),E\ra=\tfrac{1}{|\mu|^2}\left\langle E\cdot\mu,\mu\right\rangle, \qquad \mu\in
V\smallsetminus\{ 0\}, \quad E\in\sym(n).
\end{equation}
We are using $\g=\sog(n)\oplus\sym(n)$ as a Cartan decomposition, where $\sog(n)$ and $\sym(n)$ denote the subspaces of skew-symmetric and symmetric matrices, respectively.  Note that $m$ is also defined on the projective space $\PP(V)$.

\subsection{Convex subsets}\label{convex-sec}
Let $W$ be a real vector space endowed with an inner product.  A compact and convex subset $E$ of $W$ is called a {\it convex body} and a subset $F\subset E$ is said to be a {\it face} of $E$ if it is convex and for each pair of points $x,y\in E$ such that the {\it relative interior} (i.e.\ the interior as a subset of the generated affine subspace) of the segment $[x,y]$ meets $F$ one has that $[x,y]\subset F$.  An {\it extreme point} of $E$ is a point which is a face and a face $F$ of $E$ is called {\it exposed} if there exists $\alpha\in W$ such that
$$
F=\left\{ x\in E:\la x,\alpha\ra=\max\{\la y,\alpha\ra:y\in E\}\right\}.
$$
Given a subset $X\subset W$, its {\it convex hull} is defined by
$$
\CH(X):=\left\{ a_1x_1+\cdots+a_kx_k:x_i\in X, \; a_i\geq 0, \; \sum a_i=1, \; k\in\NN\right\}.
$$
Any convex body is the convex hull of its extreme points and also the disjoint union of the relative interiors of its faces.  The convex hull of two disjoint disks of same radius in $\RR^2$ is an example of a convex body with four non exposed extreme points.

If $X$ is a finite subset, say $X=\{ x_1,\dots,x_n\}$, $\CH(X)$ is called a {\it convex polytope}.  In this case, all the faces of $\CH(X)$ are exposed and it is easy to see that its relative interior is given by
$$
\CH^\circ(x_1,\dots,x_n):=\left\{ a_1x_1+\cdots+a_nx_n:a_i>0, \; \sum a_i=1\right\}.
$$
A subset $C\subset W$ is called a {\it cone} if $rx\in W$ for any $r>0$, $x\in W$.

\subsection{Convexity properties of the moment map}\label{git}
In \cite{HnzSch, BllGhgHnz}, many nice and useful results on the convexity of the image of the moment map have been obtained.  In order to apply these results to our $\Gl_n(\RR)$-representation $V=\lam$ (see Section \ref{preli}), we use the notation in such articles and set
$$
U:=\U(n), \quad U^\CC=\Gl_n(\CC), \quad Z:=\PP(\Lambda^2(\CC^n)^*\otimes\CC^n).
$$
Thus $\PP(V)$ is a $\Gl_n(\RR)$-invariant closed subset of $Z$.  For any compatible subgroup $G\subset\Gl_n(\RR)$, one has that $K:=G\cap\Or(n)$ is a maximal compact subgroup of $G$, $\ggo=\kg\oplus\pg$ is a Cartan decomposition and $G=K\exp{\pg}$, where $\pg:=\ggo\cap\sym(n)$ and $\ggo$, $\kg$ denote the Lie algebras of $G$, $K$, respectively.  Consider $\ag\subset\pg$, a maximal abelian subalgebra.  Thus the corresponding torus $A=\exp{\ag}\subset G$ is also a compatible subgroup.

The moment map $m:V\smallsetminus\{ 0\}\longrightarrow\pg$ for the $G$-action is given by composing the moment map \eqref{defmm} for the $\Gl_n(\RR)$-action with the orthogonal projection from $\sym(n)$ to $\pg$, and the one for the $A$-action, $m_\ag:V\smallsetminus\{ 0\}\longrightarrow\ag$, by projecting on $\ag$.  Let $\ag_+\subset\ag$ denote a Weyl chamber of $G$.

We now consider closed $G$-invariant subsets of $\PP(V)$.  A subset $X\subset\PP(V)$ is called {\it irreducible} if it is a real semi-algebraic subset whose real algebraic Zariski closure is irreducible (see \cite{HnzSch}).  Note that the projection on $\PP(V)$ of any orbit closure $\overline{G\cdot\mu}$ is an irreducible subset if $G$ is connected.

\begin{theorem}\label{conv1}\cite{HnzSch}
Let $X$ be a closed $G$-invariant subset of $\PP(V)$.
\begin{itemize}
\item[(i)] $m(X)\cap\ag$ is the union of finitely many convex polytopes;

\item[(ii)] $m(X)\cap\ag_+$ is a convex polytope if $X$ is irreducible.

\item[(iii)] $m_\ag(X)$ is a convex polytope if $X$ is irreducible.
\end{itemize}
\end{theorem}

In particular, $m(\overline{G\cdot\mu})\cap\ag_+$ and $m_\ag(\overline{A\cdot\mu})$ are both convex polytopes for any $\mu\in V$.  Note that part (iii) is just part (ii) applied to $G=A$ and that if $W:=N_K(\ag)/C_K(\ag)$ is the corresponding Weyl group, then for $X$ irreducible one has that
\begin{equation}\label{unionW}
m(X)\cap\ag = \bigcup\limits_{k\in W} k\cdot(m(X)\cap\ag_+).
\end{equation}

Let $X^A$ denote the set of $A$ fixed points in $X$.

\begin{theorem}\label{conv2}\cite{BllGhgHnz}
Let $X$ be a closed $G$-invariant subset of $\PP(V)$.
\begin{itemize}
\item[(i)] $m_\ag(X^A)$ is a finite set whose convex hull is $\CH(m_\ag(X))$; in particular, $\CH(m_\ag(X))$ is a convex polytope
(see \cite[Proposition 3.1]{BllGhgHnz}).

\item[(ii)] $\CH(m(X))\cap\ag=\CH(m_\ag(X))$ and $\CH(m(X)) = K\cdot \CH(m_\ag(X))$ (see \cite[Lemma 1.1]{BllGhgHnz}).

\item[(iii)] All faces of $\CH(m(X))$ are exposed (see \cite[Theorem 0.3]{BllGhgHnz}).
\end{itemize}
\end{theorem}

\section{Ricci negative derivations}\label{RNder}

Given a nilpotent Lie algebra $\ngo$, each basis $\{ e_1,\dots, e_n\}$ of $\ngo$ identifies the vector space $\ngo$ with $\RR^n$, bringing the whole setting described in Section \ref{preli}, which will be used from now on without any further mention.  When an inner product is given on $\ngo$, one can use any orthonormal basis.  In this way, the Lie bracket $\lb$ of $\ngo$ becomes a vector in $V$, the orbit $\G\cdot\lb$ consists of those Lie brackets on $\ngo$ which are isomorphic to $\lb$ and the set $\nca$ of all nilpotent Lie brackets is a $\G$-invariant algebraic subset of $V$.  Note that each $\mu\in\nca$ determines a Riemannian manifold; namely, the Lie group $(N_{\mu},\ip)$ endowed with the left-invariant metric defined by $\ip$.  A remarkable fact is that the moment map from Section \ref{mm-sec} encodes geometric information; indeed
\begin{equation}\label{mmR}
m(\mu)=\tfrac{4}{|\mu|^2}\Ricci_{\mu},
\end{equation}
where $\Ricci_\mu$ is the Ricci operator of $(N_{\mu},\ip)$ (see e.g.\ \cite{alek}).

Each $D\in\Der(\ngo)$ defines a solvable Lie algebra
$$
\sg=\RR f\oplus\ngo,
$$
given as the semi-direct sum such that $\ad{f}|_{\ngo}=D$.  If $\ip$ is an inner product on $\sg$ such that $|f|=1$ and $f\perp\ngo$, then it is easy to see using e.g. \cite[(11)]{alek} that the Ricci operator of $(\sg,\ip)$ is given by
\begin{equation}\label{Ric-half}
\Ricci=\left[\begin{array}{c|c}
-\tr{S(D)^2} & \ast \\\hline
& \\
\ast & \Ricci_\ngo+\unm[D, D^t]-\tr(D)S(D) \\ &
\end{array}\right],
\end{equation}
where $\Ricci_\ngo=\tfrac{|\lb|^2}{4}m(\lb)$ is the Ricci operator of $(\ngo,\ip)$,  $S(D):=\unm(D+D^t)$ and
$$
\la\Ricci f,X\ra=-\tr{S(D)\ad_\ngo{X}}, \qquad\forall X\in\ngo.
$$
It is easy to see that $\ast=0$ if $D$ is normal (see e.g.\ the proof \cite[Proposition 4.3]{solvsolitons}).

\begin{definition}
A derivation $D$ of a nilpotent Lie algebra $\ngo$ with $\tr{D}>0$ is said to be {\it Ricci negative} if the solvable Lie algebra $\sg=\RR f\oplus\ngo$ defined above admits an inner product of negative Ricci curvature such that $D^t=D$ and $f\perp\ngo$.
\end{definition}

We ask for the positivity of the trace in the above definition since the isomorphism class of $\sg=\RR f\oplus\ngo$ is invariant up to a nonzero scaling of $D$.  Furthermore, unimodular solvable Lie algebras do not admit Ricci negative metrics (see \cite{Dtt}), so $\tr{D}\ne 0$ is a necessary condition.  Note that any Ricci negative derivation is diagonalizable.  It easily follows from \eqref{Ric-half} that if $D>0$ (i.e.\ all its eigenvalues are positive) then $D$ is Ricci negative.  On the other hand, the only known necessary condition for a derivation $D$ to be Ricci negative is that $D$ must be positive when restricted to the center of $\ngo$ (see \cite[Theorem 2, (1)]{NklNkn}).  We consider in Section \ref{RNnil} the problem of which nilpotent Lie algebras admit a Ricci negative derivation.

The following natural question arises:

\begin{quote}
(Q1) Given a nilpotent Lie algebra $\ngo$ with a basis $\{ e_i\}$, what kind of set is the cone
$$
\left\{ D\in\Der(\ngo):D \;\mbox{is diagonal relative to}\; \{ e_i\}\; \mbox{and Ricci negative}\right\}?
$$
Is it open in the space of diagonal derivations?  Is it convex?
\end{quote}

In \cite{NklNkn,Nkl}, it was proved that this cone is open and convex for Heisenberg and filiform Lie algebras endowed with the standard bases.

\subsection{Ricci negative derivations in terms of the moment map}
Let $G_D$ denote the connected component of the identity of the centralizer subgroup of $D$ in $\G$ and let $\ggo_D$ be its Lie algebra.  Given an inner product $\ip$ on a vector space $\ngo$, we denote by $\sym(\ngo)$ the space of symmetric operators of $\ngo$ and by $\sym(\ngo)_{>0}$ the open cone of positive definite ones.  If $m$ is the moment map defined as in \eqref{defmm} by $\ip$, then the moment map $m_D$ for the $G_D$-action satisfies that $m_D(\mu)$ is the orthogonal projection of $m(\mu)$ on $\sym(\ngo)\cap\ggo_D$.  Since $m(\mu)$ commutes with any symmetric derivation of $(\ngo,\mu)$ and $D$ is a derivation of any Lie bracket in the set $\overline{G_D\cdot\lb}$, we obtain that
\begin{equation}\label{mDm}
m_D=m \quad \mbox{on} \quad \overline{G_D\cdot\lb} \quad \mbox{if} \quad D\in\sym(\ngo).
\end{equation}

\begin{theorem}\label{main2}
Let $\ngo$ be a nilpotent Lie algebra endowed with an inner product and consider $D\in\Der(\ngo)\cap\sym(\ngo)$ with $\tr{D}>0$.  Then the following conditions are equivalent:

\begin{itemize}
\item[(i)] $D$ is Ricci negative.
\item[ ]
\item[(ii)] $D\in\RR_{>0}\, m\left(G_D\cdot\lb\right) + \sym(\ngo)_{>0}$.
\item[ ]
\item[(iii)] $D\in\RR_{>0}\, m\left(\overline{G_D\cdot\lb}\right) + \sym(\ngo)_{>0}$.
\end{itemize}
\end{theorem}

\begin{proof}
Let $\ip$ denote the inner product endowing $\ngo$, which we extend to $\sg$ by setting $f\perp\ngo$ and $|f|=1$.  We first assume part (i) and denote by $\ip_1$ the Ricci negative inner product on $\sg$ such that $D^t=D$ and $f\perp\ngo$, which up to scaling can be assumed to satisfy $|f|_1=1$.  There exists $h\in\sym(\ngo)_{>0}$ such that $\ip_1|_{\ngo\times\ngo}=\la h\cdot,h\cdot\ra$, giving rise to an isometry
\begin{equation}\label{isome}
(\sg=\RR f\oplus\ngo,\ip_1) \longrightarrow (\sg_1=\RR f\oplus\ngo,\ip),
\end{equation}
where the Lie bracket of $\sg_1$ is defined by $\ad_1{f}|_{\ngo}=hDh^{-1}$, $\lb_1|_{\ngo\times\ngo}=h\cdot\lb$.  The isometry is produced by the orthogonal isomorphism sending $f$ to $f$ and each $X\in\ngo$ to $h(X)$.  Since $D$ is also symmetric with respect to $\ip_1$ we obtain that $h\in G_D$, and so the Ricci operator of $(\sg_1,\ip)$, which is also negative definite by \eqref{isome}, is given by
$$
\Ricci_1|_{\ngo} = r\, m(h\cdot\lb) - \tr(D)D,
$$
for some $r>0$ (see \eqref{Ric-half} and \eqref{mmR}), from which part (ii) follows .

Since (iii) follows trivially from (ii), it would only remains to show that part (iii) implies (i).  Assume that $D=r\, m(\mu)+E$ for some $r>0$, $E\in\sym(\ngo)_{>0}$ and that there exist $h_k\in G_D$ such that $h_k\cdot\lb$ converges to $\mu$, as $k\to\infty$.  Thus $D\in\Der(h_k\cdot\lb)$ for any $k$, $D\in\Der(\mu)$ and by scaling $\mu$ appropriately we obtain that $\Ricci_\mu-\tr(D)D$ is negative definite (see \eqref{mmR}).  This implies that $(\sg_1,\ip)$ has negative Ricci curvature if we define the Lie bracket on $\sg_1$ using $D$ and $\mu$ as above, and consequently, $(\sg_2,\ip)$ with Lie bracket defined by $D$ and $h_k\cdot\lb$ is also negatively Ricci curved for a sufficiently large $k$ by continuity.  By applying the isometry \eqref{isome}, one shows that $\la h_k\cdot,h_k\cdot\ra$ produces a Ricci negative metric on $\sg$, concluding the proof.
\end{proof}

\begin{remark}
In much the same way as in the above proof, one obtains that the solvable Lie algebra $\sg=\RR f\oplus\ngo$ admits an Einstein (non-flat) inner product such that $D^t=D$ and $f\perp\ngo$ if and only if $D\in\RR_{>0}\, m\left(G_D\cdot\lb\right) + \RR_{>0}I$.
\end{remark}

Recall that a linear operator of $\ngo$ is diagonalizable over $\RR$ if and only if it is symmetric with respect to some inner product on $\ngo$.  If instead of an inner product we fix a basis of the Lie algebra $\ngo$, then the above proposition can be rewritten as follows for diagonal derivations.

\begin{corollary}\label{main}
Let $\ngo$ be a nilpotent Lie algebra  endowed with a basis and consider  $D\in\Der(\ngo)\cap\tg^n$ with $\tr{D}>0$.  Then the following conditions are equivalent:

\begin{itemize}
\item[(i)] $D$ is Ricci negative.
\item[ ]
\item[(ii)] $D\in\RR_{>0}\, m\left(G_D\cdot\lb\right)\cap\tg^n + \tg^n_{>0}$.
\item[ ]
\item[(iii)] $D\in\RR_{>0}\, m\left(\overline{G_D\cdot\lb}\right)\cap\tg^n + \tg^n_{>0}$.
\item[ ]
\item[(iv)] $D\in\RR_{>0}\, m\left(\overline{G_D\cdot\lb}\right)\cap\ag_+^D + \tg^n_{>0}$, where $\ag_+^D\subset\tg^n$ is any Weyl chamber of $G_D$.
\end{itemize}
\end{corollary}

\begin{remark}
The cones in parts (ii)-(iv) are all open in $\tg^n$ as $\tg^n_{>0}$ is so.  Moreover, the subset in part (iv) is indeed an open and convex cone since  $m\left(\overline{G_D\cdot\lb}\right)\cap\ag^{D}_+$ is a convex polytope by Theorem \ref{conv1}, (ii).  Note that actually a Ricci negative $D$ must belong to the intersection of all the convex cones obtained by running over all the Weyl chambers.  This provides a very useful insight to work on question (Q1).
\end{remark}

\begin{proof}
If $D$ is diagonal when written in terms of the basis, then we consider the inner product on $\ngo$ making this basis orthonormal and the equivalence between (i), (ii) and (iii) follows in much the same way as the above theorem.  The only observation to make is that at the end of the proof of the fact that (i) implies (ii), since $[\Ricci_1|_\ngo,D]=0$, there exists $g\in\Or(n)\cap G_D$ such that $g\Ricci_1|_\ngo g^{-1}\in\tg^n$, and thus $g\Ricci_1|_\ngo g^{-1}= r\, m(gh\cdot\lb) - \tr(D)D$, from which part (ii) follows.

Finally, assume that part (ii) holds, say $D=rM+E$, where $r>0$, $M=m(g\cdot\lb)\in\tg^n$ for some $g\in G_D$ and $E\in\tg^n_{>0}$.  Thus $M$ belongs to some Weyl chamber $\ag_2$ of $G_D$, which has to be of the form $\ag_2=h\ag_+^Dh^{-1}$ for some $h\in\Or(n)\cap G_D$.  We therefore obtain that
$$
D = h^{-1}Dh = rh^{-1}Mh+h^{-1}Eh = rm(h^{-1}g\cdot\lb)+h^{-1}Eh,
$$
with $m(h^{-1}g\cdot\lb)\in m\left(G_D\cdot\lb\right)\cap\ag_+^D$ and $h^{-1}Eh\in\tg^n_{>0}$, from which part (iv) follows, concluding the proof.
\end{proof}

Given a nilpotent Lie algebra $\ngo$ endowed with a basis $\{ e_i\}$, we introduce the following notation:
$$
\dg:=\Der(\ngo)\cap\tg^n, \qquad \dg_{RN}:=\{ D\in\dg:D\;\mbox{is Ricci negative}\},
$$
and $T\subset\Gl_n(\RR)$ will denote the (connected) torus with Lie algebra $\tg^n$ (i.e.\ the subgroup of diagonal matrices with positive entries).

\begin{example}\label{heis3}
Let $\ngo$ be the $3$-dimensional Heisenberg Lie algebra with basis $\{ e_1,e_2,e_3\}$ and Lie bracket $[e_1,e_2]=e_3$.  We have that
$$
\dg=\{D:=\Diag(d_1,d_2,d_1+d_2):d_1,d_2\in\RR\},
$$
and if $D$ is generic (i.e.\ $d_1, d_2$ are different nonzero real numbers), then $G_D=T$, $\overline{T\cdot\lb}=\RR_{\leq 0}\lb$ and $m(\mu)=F_{12}^3$ for any $\mu=x\lb$.  Therefore, according to Corollary \ref{main}, a generic $D\in\dg$ with $\tr{D}>0$ belongs to $\dg_{RN}$ if and only if there exists $a\geq 0$ such that
$$
d_1+a>0, \quad d_2+a>0, \quad d_1+d_2-a>0,
$$
or equivalently, $d_1+d_2>a>-d_1,-d_2$.  By using $d_1+d_2>0$ we obtain that this is in turn equivalent to $2d_1+d_2>0$ and $d_1+2d_2>0$.  This implies that
$$
\dg_{RN}=\{D\in\dg:2d_1+d_2>0, \; d_1+2d_2>0\},
$$
an open and convex cone.  Indeed, since any non-generic derivation $D_0$ with positive trace belongs to the cone on the right and $T\subset G_{D_0}$, so $D_0\in\dg_{RN}$ also by Corollary \ref{main}.  \end{example}

\begin{example}\label{n4nice}
Let $\ngo$ be the $4$-dimensional nilpotent Lie algebra with basis $\{ e_1,\dots,e_4\}$ and Lie bracket
$$
[e_1,e_2]=e_3, \quad [e_1,e_3]=e_4.
$$
Since
$$
\dg=\{D:=\Diag(d_1,d_2,d_1+d_2,2d_1+d_2):d_1,d_2\in\RR\},
$$
we obtain that $D$ is generic if and only if $d_1, d_2\ne 0$ and $d_1\pm d_2\ne 0$.  In that case, $G_D=T$ and $\overline{T\cdot\lb}$ is given by the linear subspace of $V$ of nilpotent Lie brackets $\mu=\mu(x,y)$ defined by
$$
\mu(e_1,e_2)=xe_3, \quad \mu(e_1,e_3)=ye_4,  \qquad x,y\geq 0,
$$
and the moment map is given by
$$
m(\mu) = \frac{2}{|\mu|^2}\left[\begin{array}{cccc}
-(x^2+y^2)&&&\\
&-x^2&&\\
&&x^2-y^2&\\
&&&y^2
\end{array}\right] = \frac{1}{x^2+y^2}\left(x^2F_{12}^3+y^2F_{13}^4\right),
$$
which implies that $m\left(\overline{T\cdot\lb}\right)\cap\tg^4=m\left(\overline{T\cdot\lb}\right)=\CH(F_{12}^3, F_{13}^4)$.  Let us now show that
$$
\dg_{RN}=\{D\in\dg:d_1+d_2>0, \; 2d_1+d_2>0\},
$$
which is open and convex as in the above example.  It follows from Corollary \ref{main} that a generic $D\in\dg$ belongs to $\dg_{RN}$ if and only if there exist $a,b\geq 0$ such that
\begin{equation}\label{n4nice-eq}
d_1+a+b>0, \quad d_2+a>0, \quad d_1+d_2-a+b>0, \quad 2d_1+d_2-b>0.
\end{equation}
The last inequality implies that $2d_1+d_2>0$ and the condition $d_1+d_2>0$ follows by adding the last three inequalities.  To prove that these two conditions are sufficient we proceed as follows.  Note that \eqref{n4nice-eq} is equivalent to the existence of $b\geq 0$ such that
$$
2d_1+d_2>b>-d_1-a,-d_1-d_2+a,
$$
which holds if and only if there is an $a\geq 0$ such that
$$
3d_1+2d_2>a>-3d_1-d_2,-d_2,
$$
that is,  $2d_1+d_2>0$ and $d_1+d_2>0$ since $3d_1+2d_2>0$.  On the other hand, the only non-generic derivation with positive trace is $D_0=(1,-1,0,1)$ (up to scaling), and it easy to check that
$$
m(G_{D_0}\cdot\lb)\cap\tg^4 = \CH^\circ(F_{12}^3, F_{13}^4)\cup \CH^\circ(F_{24}^3, F_{34}^1).
$$
This also follows from \eqref{unionW}.  Thus $D_0$ is not Ricci negative by Corollary \ref{main}; indeed, $D_0=aF_{24}^3+bF_{35}^1+E$, $a,b\geq 0$, $E\in\tg^n_{>0}$, implies that $a>1>b$ and $b>a$, a contradiction.
\end{example}

The following corollary of Theorem \ref{main2} follows from Theorem \ref{conv1}, (iii) and provides a necessary condition for a symmetric derivation to be Ricci negative.  We denote by $\Diagg(A)$ the diagonal part of a matrix $A$.

\begin{corollary}\label{cor-main2}
Let $\ngo$ be a nilpotent Lie algebra endowed with an inner product.  If $D$ is a symmetric derivation of $\ngo$ which is Ricci negative, then relative to any orthonormal basis of $\ngo$, the diagonal part of $D$ belongs to the cone
$$
\RR_{>0}\, \Diagg\left(m\left(\overline{G_D\cdot\lb}\right)\right) + \tg^n_{>0}.
$$
\end{corollary}

Recall from Theorem \ref{conv1}, (iii) that this cone is open and convex.

\subsection{Using the convexity properties of the moment map}
The characterizations of Ricci negative derivations obtained in Corollary \ref{main} lead us to apply the results from real GIT described in Section \ref{git} to try to understand the set
$$
m\left(\overline{G_D\cdot\lb}\right)\cap\tg^n.
$$
This coincides with $m(X)\cap\ag$ in the case when $G=G_D$, $\ag=\tg^n$ and $X=\PP\left(\overline{G_D\cdot\lb}\right)$.  Recall from \eqref{mDm} that the moment maps for $G_D$ and $\Gl_n(\RR)$ coincide on $X$ in this situation.  However, the Weyl chambers for $G_D$ can be much larger.  It follows from Theorem \ref{conv1}, (ii) that $m\left(\overline{G_D\cdot\lb}\right)\cap\ag^{D}_+$ is a convex polytope for any Weyl chamber $\ag^{D}_+\subset\tg^n$ of $G_D$ and hence, as in \eqref{unionW}, one obtains that $m\left(\overline{G_D\cdot\lb}\right)\cap\tg^n$ is the union of finitely many convex polytopes by running over all Weyl chambers for $G_D$.  In particular, it is convex if $G_D=T$.

In view of the the fact that the torus $T\subset\Gl_n(\RR)$ is always contained in $G_D$ for any $D\in\dg$, we need to deepen the study of the subset
$$
m\left(\overline{T\cdot\lb}\right)\cap\tg^n.
$$
So in what follows, according to the notation in Section \ref{git}, we consider $G=\Gl_n(\RR)$, thus $\ag=\tg^n$ and $A=T$.  Recall that $m_\ag=\Diagg\circ\, m$.

For each $\mu\in V$, we define the following convex subsets of $\tg^n$,
$$
\CH_\mu := \CH\left(F_{ij}^k:c(\mu)_{ij}^k\ne 0\right), \quad \CH_\mu^\circ := \CH^\circ\left(F_{ij}^k:c(\mu)_{ij}^k\ne 0\right),
$$
where $c(\mu)_{ij}^k$ are the structural constants of $\mu$ (see Section \ref{V-sec}).  Note that $F_{ij}^k=m(\mu_{ijk})$.

\begin{lemma}\label{mtoro}
$\Diagg\left(m\left(\overline{T\cdot\mu}\right)\right) = \CH_\mu$.
\end{lemma}

\begin{proof}
For any $E\in\tg^n$ and $\lambda\in V$ one has that
$$
E\cdot\lambda = \sum \la E,F_{ij}^k\ra c(\lambda)_{ij}^k\,\mu_{ijk},
$$
so it follows from \eqref{defmm} that
\begin{equation}\label{mmdiag}
m_\ag(\lambda) =\tfrac{2}{|\lambda|^2}\sum\limits_{i<j}\left(c(\lambda)_{ij}^k\right)^2
F_{ij}^k\in\CH_\lambda, \qquad\forall\lambda\in V.
\end{equation}
In particular, $m\left(\overline{T\cdot\mu}\right)\cap\tg^n$ is contained in $\CH_\mu$ for any $\mu\in V$.

On the other hand, since
$$
c(h\cdot\mu)_{ij}^k = \frac{h_k}{h_ih_j}c(\mu)_{ij}^k, \qquad\forall h:=\Diag(h_1,\dots,h_n)\in T,
$$
one obtains that $\CH_\lambda\subset\CH_\mu$ for any $\lambda\in\overline{T\cdot\mu}$, which implies that $m_\ag(\overline{T\cdot\mu})\subset\CH_\mu$ by \eqref{mmdiag}.  But $\mu_{ijk}\in\overline{T\cdot\mu}$ for any $c(\mu)_{ik}^k\ne 0$; indeed, $e^{t\alpha}\cdot\mu$ converges to $\mu_{ij}^k$ as $t\to\infty$ for $\alpha\in\tg^n$ defined by $\alpha_r=1$ for $r=i,j$ and equal to $2$ otherwise.  This implies that if $c(\mu)_{ij}^k\ne 0$, then $F_{ij}^k\in m_\ag(\overline{T\cdot\mu})$, which is convex by Theorem \ref{conv1}, (iii), and so $\CH_\mu\subset m_\ag(\overline{T\cdot\mu})$, concluding the proof.
\end{proof}

An alternative proof of the above lemma can be given by using Theorem \ref{conv2}, (i) and the fact that the $T$ fixed points are given by
$$
\PP(V)^T=\left\{ [\mu_{ijk}]\right\}, \qquad \overline{T\cdot[\mu]}^T=\left\{[\mu_{ijk}]:c(\mu)_{ij}^k\ne 0\right\},
$$
and their $m$-images by
$$
m\left(\PP V^T\right)=\left\{ F_{ij}^k\right\}, \qquad m\left(\overline{T\cdot[\mu]}^T\right)=\left\{F_{ij}^k:c(\mu)_{ij}^k\ne 0\right\}.
$$

We now show that $\overline{T\cdot\mu}$ is actually determined by $\CH_\mu$, which is in a sense a converse to Lemma \ref{mtoro}.  For each subset $J\subset I_\mu:=\{ (i,j,k):c(\mu)_{ij}^k\ne 0\}$, consider the bracket
$$
\lambda_J:=\sum_{(i,j,k)\in J} c(\mu)_{ij}^k\,\mu_{ijk}.
$$
Note that $m_\ag\left(\overline{T\cdot\lambda_J}\right) = \CH\left(F_{ij}^k:(i,j,k)\in J\right)$ (see \eqref{mtoro}).  We recall that some basic convex geometry terminology was given in Section \ref{convex-sec}.

\begin{lemma}\label{A-deg}
The closure of the orbit $T\cdot\mu$ is given by
$$
\overline{T\cdot\mu} =\left\{\lambda_J:\CH\left(F_{ij}^k:(i,j,k)\in J\right)\;\mbox{is a face of}\; \CH_\mu\right\}.
$$
Moreover, if $c(\mu)_{ij}^k\ne 0$, then $\mu_{ij}^k\in\overline{T\cdot\mu}$ and $F_{ij}^k$ is an extreme point of $\CH_\mu$.
\end{lemma}

\begin{proof}
It follows from \cite[Theorem 1.1, (ii)]{BhmLfn} that a Lie bracket $\lambda$ belongs to $\overline{T\cdot\mu}$ if and only if there exists $\alpha\in\tg^n$ such that $e^{t\alpha}\cdot\mu$ converges to $\lambda$, as $t\to\infty$.  In particular, $\lambda=\lambda_J$ for some $J\subset I_\mu$, and such convergence is equivalent to $\la\alpha,F_{ij}^k\ra=0$ for any $(i,j,k)\in J$ and negative otherwise.  But the existence of an $\alpha\in\tg^n$ with such properties is necessary and sufficient to have that $\CH\left(F_{ij}^k:(i,j,k)\in J\right)$ is a face of $\CH_\mu$, concluding the proof.
\end{proof}

\begin{corollary}
$\Diagg\left(m\left(T\cdot\lb\right)\right) = \CH_\mu^\circ$.
\end{corollary}

The situation in Example \ref{n4nice} concerning convexity properties can drastically change if we consider a different basis for that Lie algebra, as next example shows.

\begin{example}\label{n4nonice}
If the Lie bracket is defined by
$$
[e_1,e_2]=e_3+e_4, \quad [e_1,e_3]=e_4,
$$
then $\overline{T\cdot\lb}$ is given by the brackets $\mu=\mu(x,y,z)$ such that
$$
\mu(e_1,e_2)=xe_3+ye_4, \quad \mu(e_1,e_3)=ze_4,  \qquad x,y,z\geq 0.
$$
According to Lemma \ref{A-deg}, $\overline{T\cdot\lb}$ consists, besides of the open and dense orbit $T\cdot\lb$ ($x,y,z>0$), of other six nonzero $T$-orbits defined by the (external) faces of the triangle $\CH_{\lb}=\CH(F_{12}^3, F_{12}^4, F_{13}^4)$.  The moment map is given by
\begin{align*}
m(\mu) =& \frac{2}{|\mu|^2}\left[\begin{array}{cccc}
-x^2-y^2-z^2&&&\\
&-x^2-y^2&-yz&\\
&-yz&x^2-z^2&xy\\
&&xy&y^2+z^2
\end{array}\right] \\
=& \frac{1}{x^2+y^2+z^2}\left(x^2F_{12}^3+y^2F_{12}^4+z^2F_{13}^4-yzF_5+xyF_6\right),
\end{align*}
where
$$
F_5:=\left[\begin{smallmatrix}
0&&\\
&0&1&\\
&1&0&\\
&&&0
\end{smallmatrix}\right], \qquad
F_6:=\left[\begin{smallmatrix}
0&&\\
&0&&\\
&&0&1\\
&&1&0
\end{smallmatrix}\right].
$$
Thus $m(T\cdot\lb)\cap\tg^4=\emptyset$,
$$
m\left(\overline{T\cdot\lb}\right)\cap\tg^4=\CH(F_{12}^3, F_{13}^4)\cup \left\{F_{12}^4\right\},
$$
and recall from Lemma \ref{mtoro} that $\Diagg\left(m\left(\overline{T\cdot\lb}\right)\right)=\CH(F_{12}^3, F_{12}^4, F_{13}^4)$.  Nevertheless, it easily follows from Corollary \ref{main} that $\dg_{RN}=\{ (0,d,d,d):d>0\}$ by using only $F_{12}^4$.
\end{example}

\subsection{Nilpotent Lie algebras with a nice basis}
The better behavior of $m\left(\overline{T\cdot\lb}\right)\cap\tg^4$ in Example \ref{n4nice} compared to what happened in Example \ref{n4nonice} is due to special properties of the basis chosen.

\begin{definition}\label{nice-def}
A basis $\{ e_1,\dots,e_n\}$ of a Lie algebra is said to be {\it nice} if every bracket $[e_i,e_j]$ is a scalar multiple of some element $e_k$ in the basis and two different brackets $[e_i,e_j]$, $[e_r,e_s]$ can be a nonzero multiple of the same $e_k$ only if $\{ i,j\}$ and $\{ r,s\}$ are either equal or disjoint.
\end{definition}

\begin{lemma}\label{nice}\cite{nicebasis}
The following conditions are equivalent:
\begin{itemize}
\item[(i)] $m\left(\overline{T\cdot\mu}\right)\cap\tg^n = \CH_\mu$.

\item[(ii)] $\{ e_i\}$ is a nice basis for $\mu$.

\item[(iii)] $m\left(T\cdot\mu\right)\subset\tg^n$.
\end{itemize}
\end{lemma}

\begin{proof}
The equivalence between parts (ii) and (iii) is precisely the result proved in \cite{nicebasis}.  Part (i) follows from (iii) and \eqref{mtoro}, so we only need to prove that part (i) implies (iii).  If $h\in T$, then $m_\ag(h\cdot\mu)=m(g\cdot\mu)$ for some $g\in T$ by (i) and \eqref{mtoro}.  But this implies that $h=tg$ for a nonzero $t\in\RR$ since $m_\ag:T\cdot[\mu]\longrightarrow m_\ag(T\cdot\mu)$ is a diffeomorphism (see \cite[Proposition 3]{HnzStt}) and thus $m(h\cdot\mu)\in\tg^n$, concluding the proof.
\end{proof}

The following result was proved in \cite[Section 4]{Nkl2}.

\begin{corollary}\label{niceD}
If $\ngo$ is a nilpotent Lie algebra with a nice basis, then $\Diagg(D)\in\Der(\ngo)$ for any $D\in\Der(\ngo)$.
\end{corollary}

\begin{proof}
For any $h\in T$ one has that $\tr{\Diagg(D)m(h\cdot\lb)} = \tr{Dm(h\cdot\lb)} = 0$, thus $\tr{\Diagg(D)F_{ij}^k}=0$ for each $c_{ij}^k\ne 0$ by Lemma \ref{A-deg}, that is, $\Diagg(D)\in\Der(\ngo)$.
\end{proof}

The above lemma together with Corollary \ref{main} also give the following.

\begin{corollary}\label{cor-nice}
Let $\ngo$ be a nilpotent Lie algebra endowed with a nice basis.  Then the open convex cone in $\dg$ given by
$$
\left(\RR_{>0}\CH_{\lb} + \tg^n_{>0}\right)\cap \{ D\in\dg:\tr{D}>0\},
$$
is contained in $\dg_{RN}$.
\end{corollary}

It follows from Theorem \ref{conv2}, (ii) that
$$
\CH\left(m\left(\overline{G_D\cdot\lb}\right)\right) = \Diag\left(m\left(\overline{G_D\cdot\lb}\right)\right),
$$
a convex polytope.  However, both $m\left(\overline{T\cdot\mu}\right)\cap\tg^n$ and $m\left(\overline{G_D\cdot\mu}\right)\cap\tg^n$ can be tricky subsets if the basis is not nice, as next example shows.

\begin{example}\label{n5nonice}
Let $\ngo$ be the $5$-dimensional nilpotent Lie algebra with basis $\{ e_1,\dots,e_5\}$ and Lie bracket
$$
[e_1,e_2]=e_3+e_4, \quad [e_1,e_3]=e_5, \quad [e_1,e_4]=e_5.
$$
It is easy to see that if $D$ is generic, then
$$
G_D=G:=\left\{\left[\begin{array}{cccc}
h_1&&&\\
&h_2&&\\
&&H&\\
&&&h_5
\end{array}\right]: H\in\Gl_2^+(\RR), \quad h_i>0\right\}.
$$
We consider $G$ acting on the cone $C\subset V$ of nilpotent Lie brackets $\mu=\mu(x,y,z,w)$ defined by
$$
\mu(e_1,e_2)=xe_3+ye_4, \quad \mu(e_1,e_3)=ze_5, \quad \mu(e_1,e_4)=we_5, \qquad x,y,z,w\geq 0.
$$
The moment map $m:C\smallsetminus\{ 0\}\longrightarrow\pg$ is given by
\begin{align*}
m(\mu) =& \frac{2}{|\mu|^2}\left[\begin{array}{ccccc}
-(x^2+y^2+z^2+w^2)&&&&\\
&-x^2-y^2&&&\\
&&x^2-z^2&xy-zw&\\
&&xy-zw&y^2-w^2&\\
&&&&z^2+w^2
\end{array}\right] \\
=& \frac{1}{x^2+y^2+z^2+w^2}\left(x^2F_{12}^3+y^2F_{12}^4+z^2F_{13}^5+w^2F_{14}^5+(xy-zw)F\right),
\end{align*}
where
$$
F:=\left[\begin{smallmatrix}
0&&&\\
&0&&&\\
&&0&1&\\
&&1&0&\\
&&&&0
\end{smallmatrix}\right].
$$
It is easy to check that $C\smallsetminus\{ 0\}$ is the disjoint union of three orbits: $G\cdot\lb$, $G\cdot\mu_{123}$ and $G\cdot\mu_{135}$; and the first one is given by
$$
G\cdot\lb=\{\mu:xz+yw\ne 0\}, \qquad \overline{G\cdot\lb}=C.
$$
This implies that
\begin{align}
m\left(\overline{G\cdot\lb}\right)\cap\tg^5 =& \{ aF_{12}^3+bF_{12}^4+cF_{13}^5+dF_{14}^5: a,b,c,d\geq 0, \label{image1}\\
&\quad a+b+c+d=1, \quad ab=cd\}  \notag\\
=& \{ \Diag(-1,-a-b,a-c,b-d,c+d)\in\tg^5:a,b,c,d\geq 0, \label{image2}\\
&\quad a+b+c+d=1, \quad ab=cd\}. \notag
\end{align}
Since
$$
F_{12}^3-F_{12}^4=F_{14}^5-F_{13}^5 = \Diag(0,0,1,-1,0) \perp (0,-1,1,1,-1) = F_{12}^4-F_{13}^5 = F_{12}^3-F_{14}^5,
$$
these four points $\left\{ F_{12}^3, F_{12}^4, F_{13}^5, F_{14}^5\right\}$ in $\tg^5$ are the vertices of a rectangle with center $F_0:=(-1,-\unm,0,0,\unm)$, which is precisely $\CH_{\lb}=\Diagg\left(m\left(\overline{G\cdot\lb}\right)\right)$.  It is not so hard to see by using \eqref{image1} that $m\left(\overline{G\cdot\lb}\right)\cap\tg^5$ is the union of two triangles, given by the convex hulls of $\{F_{12}^4,F_{13}^5,F_0\}$ and $\{F_{12}^3,F_{14}^5,F_0\}$, respectively.  Such computation becomes clearer if one translates everything to the origin by subtracting the vector $F_0$ from all the vectors involved.  The two Weyl chambers are given by
$$
(\ag_+)_1=\{\Diag(a_1,\dots,a_5):a_3\leq a_4\}, \qquad (\ag_+)_2=\{\Diag(a_1,\dots,a_5):a_3\geq a_4\},
$$
thus the convex polytope $m\left(\overline{G\cdot\lb}\right)\cap(\ag^\dg_+)_1$ can be obtained by adding to \eqref{image1} or \eqref{image2} the condition $a+d\leq b+c$, and so it coincides with the triangle $\{F_{12}^4,F_{13}^5,F_0\}$.  The other triangle corresponds to the other Weyl chamber.

Concerning $T$-orbits, it follows from Lemma \ref{A-deg} that $\overline{T\cdot\lb}$ consists of the open and dense orbit  $T\cdot\lb=\{\mu:xz=yw\ne 0\}$ and other eight nonzero $T$-orbits corresponding to the edges and vertexes of the rectangle $\CH_{\lb}$.  From \eqref{image1} we obtain that
\begin{align}
m\left(\overline{T\cdot\lb}\right)\cap\tg^5 =& \{ aF_{12}^3+bF_{12}^4+cF_{13}^5+dF_{14}^5: a,b,c,d\geq 0, \label{image3}\\
&\quad a+b+c+d=1, \quad ab=cd, \quad ac=bd\}.  \notag
\end{align}
It is easy to check that these conditions hold if and only if either $a=d=0$, or $b=c=0$, or $a=d$ and $b=c$, hence $m\left(\overline{T\cdot\lb}\right)\cap\tg^5$ is the union of three segments, $\overline{F_{12}^3F_{14}^5}$, $\overline{F_{12}^4F_{13}^5}$ and the one having as extremes their middle points.  The interior of this last segment coincide with $m\left(T\cdot\lb\right)\cap\tg^5$.
\end{example}

\begin{remark}
In the above example, in terms of the notation in \cite{HnzSch}, we have the compatible group $G=G_\dg$ acting on the closed subset $X:=\PP(W)$, with Cartan decomposition
$$
\ggo=\RR^2\oplus\glg_2(\RR)\oplus\RR, \qquad \pg=\RR^2\oplus\sym(2)\oplus\RR, \qquad \ag=\tg^5, \qquad A=T^0.
$$
The $A$ fixed points in $X$ are exactly $[E_{21}], [E_{31}], [E_{42}], [E_{43}]$, i.e.\ the weight vectors for the $A$-representation $W$, which have as $m$-images the matrices $F_{12}^3, F_{12}^4, F_{13}^5, F_{14}^5$, respectively.  It follows that $m(X)\cap\ag$ is not a union of convex hulls of subsets of $m$-images of $A$ fixed points in $X$, in spite $X$ is irreducible, as asserted in the first theorem in the introduction of \cite{HnzSch}.
\end{remark}

\subsection{An application in low dimension}

Any nilpotent Lie algebra of dimension $\leq 6$ has a positive derivation.  In dimension $7$, the first examples such that any derivation is nilpotent appear.  Note that these algebras do not admit any non-nilpotent solvable extension, they are called {\it characteristically nilpotent}.  An inspection of the classification of nilpotent Lie algebras of dimension $\leq 7$ (see e.g.\ \cite{Mgn} or \cite{Frn} and the references therein), which includes more than one hundred algebras and some continuous families, gives that among those which are not characteristically nilpotent only four do not admit a positive derivation.  We now apply the results obtained in this section to show that each of these four nilpotent Lie algebras has a solvable extension admitting a Ricci negative metric.

\begin{theorem}\label{dim7}
Any nilpotent Lie algebra of dimension $\leq 7$ which is not characteristically nilpotent admits a Ricci negative derivation.
\end{theorem}

\begin{proof}
The four $7$-dimensional algebras mentioned above are defined by
\begin{align}
&[e_1, e_2] = e_4, \; [e_1, e_4] = e_5, \; [e_1, e_5] = e_6, \; [e_1, e_6] = e_7, \; [e_2, e_3] = e_5+e_7, \label{alg1} \\
&[e_3, e_4] =-e_6, \; [e_3, e_5] = -e_7, \quad D=\Diag(0, 1, 0, 1, 1, 1, 1). \notag \\
&[e_1, e_2] = e_4, \; [e_1, e_4] = e_5, \; [e_1, e_5] = e_6, \; [e_1, e_6] = e_7, \; [e_2, e_3] = e_6+e_7, \label{alg2} \\
&[e_3, e_4] =-e_7, \quad D=\Diag(0, 1, 0, 1, 1, 1, 1). \notag \\
&[e_1, e_2] = e_3, \; [e_1, e_3] = e_4, \; [e_1, e_5] = e_6, \; [e_2, e_3] = e_5, \; [e_2, e_4] = e_6, \label{alg3} \\
&[e_2, e_5] =e_7, [e_2, e_6] =e_7, \; [e_2, e_5] =-e_7,  \quad D=\Diag(0, 1, 1, 1, 2, 2, 3). \notag \\
&[e_1, e_2] = e_3, \; [e_1, e_3] = e_4, \; [e_1, e_4] = e_5, \; [e_1, e_6] = e_7, \; [e_2, e_3] = e_6, \label{alg4} \\
&[e_2, e_4] =e_7, [e_2, e_5] =e_7, \; [e_3, e_4] =-e_7,  \quad D=\Diag(0, 1, 1, 1, 1, 2, 2). \notag
\end{align}

We will prove that the derivation $D$ given in each case is Ricci negative, from which the theorem follows.  For the algebra \eqref{alg1} , we can use $\alpha:=\Diag(-1,0,-2,-1,-2,-3,-4)$ to show that $\CH(F_{ij}^k:c_{ij}^k\ne 0,\; F_{ij}^k\ne F_{23}^7)$ is a face of $\CH_{\lb}$ which corresponds to the degeneration $\lambda:=\lb-\mu_{237}\in \overline{T\cdot\lb}$ (see Lemma \ref{A-deg}).  Note that $\lambda$ is nice, and we have that $D-M>0$ for
$M:=\unm(F_{12}^4+F_{23}^5)\in \CH_\lambda$.  Thus $D$ is Ricci negative by Corollary \ref{cor-nice}.  The case \eqref{alg2} follows in identical way by setting $\alpha:=\Diag(-1,0,-3,-1,-2,-3,-4)$ and $M:=\unm(F_{12}^4+F_{23}^6)$.

For the algebra \eqref{alg3}, we use that $\mu_{156}\in \overline{T\cdot\lb}$ (see Lemma \ref{A-deg}), hence
$$
M:=F_{15}^6\in m\left(\overline{T\cdot\lb}\right)\cap\tg^7 \subset m\left(\overline{G_D\cdot\lb}\right)\cap\tg^7,
$$
and so $D$ is Ricci negative by Corollary \ref{main} since $D-M>0$.   Finally, in the case of \eqref{alg4} one can use $\mu_{167}$, concluding the proof.
\end{proof}

It is shown in Proposition \ref{ex1ex2ex5}, (i) that the above theorem already fails in dimension $8$.

\section{Ricci negative nilradicals}\label{RNnil}

In this section, we consider the following question:
\begin{quote}
(Q2)  Which nilpotent Lie algebras can be the nilradical of some solvable Lie algebra admitting a Ricci negative metric?
\end{quote}
We call such a Lie algebra a {\it Ricci negative nilradical} (RN-nilradical for short).  The name is motivated by Einstein nilradicals (see e.g.\ the survey \cite{cruzchica}).  The existence of a positive derivation (i.e.\ the real part of each eigenvalue is positive) is sufficient to be a RN-nilradical (see \cite[Theorem 2, (1)]{NklNkn}); in particular, any $2$-step nilpotent Lie algebra is a RN-nilradical.  Furthermore, it follows from Theorem \ref{dim7} that any non-characteristically nilpotent nilpotent Lie algebra of dimension $\leq 7$ is a RN-nilradical.

On the other hand, a first necessary condition on a nilpotent Lie algebra to be a RN-nilradical is the existence of a derivation with nonzero trace.  This follows from the fact that unimodular solvable Lie algebras do not admit Ricci negative metrics (see \cite{Dtt}).  In what follows, we recall some algebraic notions and facts related to such condition.

Let $\ngo$ be a real nilpotent Lie algebra.  Given $D\in\Der(\ngo)$, consider the additive Jordan decomposition for $D$ given by
$$
D=D^s+D^n, \qquad [D^s,D^n]=0, \qquad D^s=D^{\RR}+D^{\im\RR}, \qquad [D^{\RR},D^{\im\RR}]=0,
$$
where $D^s$ is semisimple (i.e.\ diagonalizable over $\CC$), $D^n$ is nilpotent, $D^{\RR}$ is {\it real semisimple} (i.e.\ diagonalizable over $\RR$) and $D^{\im\RR}$ is semisimple with only imaginary eigenvalues.  It is well-known that $D^s,D^n,D^{\RR},D^{\im\RR}\in\Der(\ngo)$.  Note that $\Spec(D)=\Spec(D^s)$ and $\Rea\Spec(D)=\Spec(D^{\RR})$.

A maximal abelian subspace of real semisimple derivations is called a {\it maximal torus} and is known to be unique up to conjugation by automorphisms; its dimension is called the {\it rank} of $\ngo$ and will be denoted by $\rank(\ngo)$.  Recall that a Lie algebra is said to be {\it characteristically nilpotent} if it has only nilpotent derivations (i.e.\ its complex rank is zero).  The first example was found in \cite{DxmLst} sixty years ago.

It is proved in \cite{Nkl2} that for any Lie algebra $\ngo$, there exists a real semisimple $\phi_\ngo\in\Der(\ngo)$ such that $\tr{\phi_\ngo D}=\tr{D}$ for all $D\in\Der(\ngo)$.  Such special derivation, which is unique up to automorphism conjugation, is called a {\it pre-Einstein} derivation.  Note that $\phi_\ngo=0$ if and only if $\tr{D}=0$ for any $D\in\Der(\ngo)$, so $\phi_\ngo\ne 0$ is a first obstruction for $\ngo$ to be a RN-nilradical.

It clearly follows that
\begin{quote}
  $\ngo$ characteristically nilpotent $\quad\Rightarrow\quad$ $\rank(\ngo)=0$ $\quad\Rightarrow\quad$ $\phi_\ngo=0$.
\end{quote}
An obvious natural question is whether the converse assertions hold.  Curiously enough, we could not find any answer in the literature. Examples \ref{ex10} and \ref{ex3} below show that the converse assertions are both false.

The Lie algebras under consideration usually have a natural diagonal derivation $D$. To study the other derivations, we will consider other real semisimple derivations $D^\prime$ which commute with the given derivation $D$. By taking quotients by invariant ideals and using the low dimensional classifications in \cite{Mgn}, which contains the full description of derivations, we find information about the general derivation and about the rank of the Lie algebra. When we define a derivation $D: \ngo \to \ngo$, we will sometimes only define its value on generators of the Lie algebra $\ngo$. In these cases, it is left to the reader to check that it indeed defines a derivation on the whole Lie algebra $\ngo$.

The existence of a nice basis on a nilpotent Lie algebra $\ngo$ (see Definition \ref{nice-def}) makes the computations concerning derivations more manageable.  Indeed, if $D^\prime \in\Der(\ngo)$, then the linear map defined by the diagonal of the matrix of $D^\prime$ with respect to a nice basis is also a derivation (see Corollary \ref{niceD}). Since the given derivation $D$ is already diagonal and $D^\prime$ commutes with $D$, the diagonal of $D^\prime$ will also commute with $D$. In some cases we will conclude that the diagonal of $D^\prime$ is equal to $\lambda D$ for some $\lambda \in \mathbb{R}$ and since $D^\prime$ is real semisimple, this will imply that $D^\prime$ is equal to its diagonal.

We note that all the examples provided in this section are written in terms of a nice basis with the only exception of Proposition \ref{ex1ex2ex5}, (i).

\begin{example}\label{ex10}
Consider the Lie algebra $\ngo$ with basis vectors $X_1, \dots, X_5, Y_1, \ldots, Y_5, Z$ and brackets defined as
\begin{align*}
[X_1, X_2] &= X_3 &[X_1, X_3] &= X_4 &[X_1, X_4] &= X_5 &[X_2,X_3] &= X_5\\
[Y_1, Y_2] &= -X_3 &[Y_1, Y_3] &= -X_4 &[Y_1, Y_4] &= -X_5 &[Y_2,Y_3] &= -X_5\\
[X_1, Y_2] &= Y_3 &[X_1, Y_3] &= Y_4 &[X_1, Y_4] &= Y_5 &[X_2,Y_3] &= Y_5\\
[Y_1, X_2] &= Y_3 &[Y_1, X_3] &= Y_4 &[Y_1, X_4] &= Y_5 &[Y_2,X_3] &= Y_5\\
[X_1, Y_1] &= Z &[X_2, Y_2] &= Z. & & & &
\end{align*}
It is straightforward to check that the Jacobi identity holds. Let $\Diag(d_1, \ldots, d_{11})$ be a diagonal derivation. The first four brackets show that $d_i = i d_1$ for $1 \leq i \leq 5$. From the brackets which lead to $Y_3$, we find that $d_1 + d_7 = d_6 + d_2$. The last two brackets leading to $Z$ imply that $d_1 + d_6 = d_2 + d_7$, from which we conclude that $d_1 = 0$. The other brackets then easily give that $\Diag(d_1, \ldots, d_{11}) = 0$.

However, since $D$ defined by $D(X_i) = i Y_i, D(Y_i) = -i X_i$ for all $1 \leq i \leq 5$ and $D(Z) = 0$ is a derivation, we conclude that $\ngo$ is not characteristically nilpotent.  Note that the basis is nice, so the diagonal of every derivation is again a derivation.  If $D^\prime$ is any real semisimple derivation which commutes with $D$, then its diagonal is equal to $0$ and hence it is of the form $D^\prime(X_i) = \lambda_i Y_i, D^\prime(Y_i) = \lambda_i X_i$ for $1 \leq i \leq 5$ and $D^\prime(Z) = \mu Z$. A similar computation as above shows that $\lambda_i = i \lambda_1$ and $\mu = 0$, which imlies that $D^\prime = \lambda_1 D$. We conclude that $\ngo$ is of complex rank $1$.

To see that $\rank(\ngo)=0$, take any diagonalizable derivation $D^\prime: \ngo \to \ngo$. Since the basis is nice, the diagonal of $D^\prime$ is again equal to $0$. Let $D^\prime(X_i) = \lambda_i Y_i + U_i$ and $D^\prime(Y_i) = \mu_i X_i + V_i$ for $i = 1,2$ and with $U_i, V_i$ a linear combination of the other basis vectors. Let $\mg = \langle X_4, X_5, Y_4, Y_5, Z \rangle$, which is invariant under $D^\prime$ as the sum of $Z(\ngo)$ and $[[\ngo,\ngo],\ngo]$. The relations leading to $X_3$ then show that
\begin{align*}
D^\prime(X_3) &= (\lambda_1 + \lambda_2) Y_3 + \mg = - (\mu_1 + \mu_2) Y_3 + \mg \\
D^\prime(Y_3) &= (- \lambda_1 + \mu_2) X_3 + \mg = (\mu_1 - \lambda_2) X_3 + \mg
\end{align*}
and hence $\lambda_i = - \mu_i$ for $i = 1, 2$. Since $D^\prime$ is diagonalizable over $\mathbb{R}$, this implies that $\lambda_1 = \lambda_2 = 0$. As the rank of $\ngo$ over $\mathbb{C}$ is equal to $1$, the derivation $D^\prime$ is conjugate over $\mathbb{C}$ to a multiple of $D$, which implies that in fact $D^\prime = 0$. We conclude that $\ngo$ has real rank $0$ but complex rank $1$.

\end{example}

The above is therefore an example of a nilpotent Lie algebra which is neither characteristically nilpotent nor a RN-nilradical.  The following example shows that the existence of a nonzero diagonalizable derivation is not sufficient to be a RN-nilradical either.

\begin{example}\label{ex3}
Let $\ngo$ be the Lie algebra of dimension $11$ with basis $X_1, \ldots, X_5, Y_1, \ldots Y_5, Z$ and bracket
\begin{align*}
[X_1, X_2] &= X_3  &[X_1, X_3] &= X_4 &[X_1, X_4] &= X_5 &[X_2, X_3] &= X_5 \\
[Y_1, Y_2] &= Y_3  &[Y_1, Y_3] &= Y_4 &[Y_1, Y_4] &= Y_5 &[Y_2, Y_3] &= Y_5 \\
[X_1, Y_1] &= Z    &[X_2, Y_2] &= Z. & & & &
\end{align*}
Consider the derivation $D$ given by $D(X_i) = i X_i$, $D(Y_j) = - j Y_j$ and $D(Z) = 0$ and let $D^\prime$ be a real semisimple derivation which commutes with $D$. Then $D^\prime$ is diagonal in this basis, since $D$ has $11$ different eigenvalues. From the first four brackets between the $X_i$ it follows that $D^\prime(X_i) = i \lambda X_i$ for some $\lambda \in \mathbb{R}$. Similarly for the $Y_j$ we find that $D^\prime(Y_j) = j \mu Y_j$ for some $\mu \in \mathbb{R}$. The last two brackets, leading to the vector $Z$, imply that $\lambda + \mu = 2 \lambda + 2 \mu$ or equivalently that $D^\prime = \lambda D$. We conclude that $\rank(\ngo)=1$ and however, every derivation has trace $0$, i.e.\ $\phi_\ngo=0$.
\end{example}

All this suggests a reformulation of question (Q2) by adding the condition $\phi_\ngo\ne 0$ on the nilpotent Lie algebras involved.  The only other known necessary condition for being a RN-nilradical was obtained in \cite{NklNkn}: there must exist a derivation $D$ with $\tr{D}>0$ whose restriction to the center $\zg(\ngo)$ of $\ngo$ is positive, in the sense that $D^{\RR}|_{\zg(\ngo)}>0$, or equivalently, the eigenvalues of $D|_{\zg(\ngo)}$ have all positive real part (see \cite[Theorem 2, (1)]{NklNkn}).  Using this obstruction, we now prove that $\phi_\ngo\ne 0$ is still not sufficient.  More precisely, the following proposition shows that the two sufficient conditions to be a RN-nilradical mentioned above (i.e.\ $2$-step and $\dim{\ngo}\leq 7$, $\ngo$ non-characteristically nilpotent) are actually sharp.  Furthermore, we found a curve of nilpotent Lie algebras which are not RN-nilradicals.  Recall that any real semisimple derivation belongs to some maximal torus, hence it is always conjugate to some derivation in a given maximal torus.

\begin{proposition}\label{ex1ex2ex5}
There exist nilpotent Lie algebras such that $\phi_\ngo\ne 0$ but any real semisimple derivation has a zero eigenvalue on the center and with the following properties:
\begin{itemize}
\item[(i)] $\dim{\ngo}=8$, $\ngo$ is $5$-step nilpotent, the dimensions of the descendent central series are $(8,5,4,3,1)$, $\dim{\zg(\ngo)}=2$, $\rank(\ngo)=1$ and $\Diag(0,1,0,1,1,1,1,0)\in\Der(\ngo)$.

\item[(ii)] $\dim{\ngo}=10$, $\ngo$ is $3$-step nilpotent with descendent central series dimensions $(10,6,2)$, $\dim{\zg(\ngo)}=3$, $\rank(\ngo)=2$ and $$\Diag(0,0,0,1,1,1,0,0,0,0), \; \Diag(0,0,0,0,0,0,1,1,1,1)\in\Der(\ngo).$$

\item[(iii)] A continuous family of pairwise non-isomorphic $13$-dimensional $6$-step nilpotent Lie algebras such that $\dim{\zg(\ngo)}=3$, $\rank(\ngo)=1$ and
    $$
    \Diag(1,2,3,4,5,6,7,-1,-2,-3,-4,-5,0)\in\Der(\ngo).
    $$
\end{itemize}
\end{proposition}

\begin{proof} Part (i).  Consider the Lie algebra $\ngo$ of dimension $8$ with basis $X_1, \ldots, X_7, Y$ and bracket
\begin{align*}
[X_1, X_2] &= X_4 &[X_1, X_4] &= X_5 &[X_1, X_5] &= X_6	 & [X_1, X_6] &= X_7\\
[X_2, X_3] &= X_6 + X_7  &[X_3, X_4] &= -X_7  &[X_1,X_3] &= Y.
\end{align*}
The center is $\zg(\ngo)=\la X_7,Y\ra$.  Note that the Lie algebra $\ngo_X = \faktor{\ngo}{\langle Y \rangle}$ has rank $1$, as it is equal to the Lie algebra $\mathcal{G}_{7,1.01(ii)}$ of \cite{Mgn}. Consider the derivation $D: \ngo \to \ngo$ given by $D(X_1) = D(X_3) = 0$ and $D(X_2) = X_2$. Every derivation $D^\prime$ which commutes with $D$ satisfies $D^\prime(Y) = \mu Y$, since $\langle Y \rangle$ is equal to the intersection of $[\ngo,\ngo]$ and the eigenspace of $D$ of eigenvalue $0$.  So, by considering the map induced by $D^\prime$ on $\faktor{\ngo}{\langle Y \rangle}$ one sees that  $
D^\prime(X_1), D^\prime(X_3) \in \langle Y \rangle$.
In particular, \begin{align*} D^\prime(Y) = [D^\prime(X_1),X_3] + [X_1,D^\prime(X_3)] = 0. \end{align*} If $D^\prime$ is real semisimple, then $D^\prime(X_1) = D^\prime(X_3) = 0$ and hence $D^\prime = \lambda D$ for some $\lambda \in \mathbb{R}$. So every real semisimple derivation has eigenvalue $0$ on the center and the proposition follows.

\vspace{.3cm}
\noindent
Part (ii).  Let $\ngo$ be the Lie algebra with basis $X_1, X_2, X_3, Y_1, Y_2, Y_3, Z_1, Z_2, Z_3, Z_4$ and bracket
\begin{align*}
[X_1, Y_1] &= Y_2 &[X_1, Y_2] &= Y_3 &[X_2, Y_1] &= Y_3\\
[X_1, Z_1] &= Z_2 &[X_2, Z_1] &= Z_3 &[X_1, Z_2] &= Z_4\\
[X_2, Z_3] &= Z_4 &[X_1, X_2] &= X_3. & &
\end{align*}
The center is generated by $X_3,Y_3,Z_4$.  Let $D$ be the derivation given by $D(X_i) = 0$, $D(Y_i) = Y_i$, $D(Z_j) = 2 Z_j$ for all $i \in \{1,2,3\}, j \in \{1,2,3,4\}$. If $D^\prime$ is a real semisimple derivation commuting with $D$, then $D^\prime(X_3) = \lambda_3 X_3$ for some $\lambda_3 \in \mathbb{R}$ since $X_3$ spans the intersection of the eigenspace of $D$ for eigenvalue $0$ and $[\ngo,\ngo]$. We will demonstrate that $\lambda_3 = 0$, which implies the proposition.

Note that, since the basis is nice, we can assume that $D^\prime$ is a diagonal derivation. Write $D(X_i) = \lambda_i X_i$, $D(Y_1) = \mu Y_1$ and $D(Z_1) = \nu Z_1$. By applying $D^\prime$ to the second and the third equation, we get $$2 \lambda_1 + \mu  = \lambda_2 + \mu.$$ Similarly, by applying $D^\prime$ to the sixth and seventh equation we get $$\nu + 2 \lambda_1 = \nu + 2 \lambda_2.$$ Hence $\lambda_1=\lambda_2 = 0$ and therefore also $\lambda_3 = \lambda_1 + \lambda_2 = 0$. The other parts follow immediately.

\vspace{.3cm}
\noindent
Part (iii).  Finally, consider the Lie algebra $\ngo_t$ with basis $X_1, \ldots, X_7, Y_1, \ldots, Y_5, Z$ and bracket
\begin{align*}	
[X_1, X_2] &= X_3 &[X_1, X_3] &= X_4 &[X_1, X_4] &= X_5 \\
[X_1, X_5] &= X_6 &[X_1, X_6] &= X_7 &[X_2, X_3] &= X_5 \\
[X_2, X_4] &= X_6 &[X_2, X_5] &= t X_7 &[X_3, X_4] &= (1-t) X_7\\
[Y_1, Y_2] &= Y_3 &[Y_1, Y_3] &= Y_4 & &  \\
[Y_1, Y_4] &= Y_5 &[Y_2, Y_3] &= Y_5 & & \\
[X_1, Y_1] &= Z &[X_2, Y_2] &= Z.
\end{align*}
The center is $\zg(\ngo)=\la X_7,Y_5,Z\ra$.  Let $D$ be the derivation given by $D(X_i) = i X_i$, $D(Y_j) = - j Y_j$ and $D(Z) = 0 $. Similarly as in Example \ref{ex3} one can show that this Lie algebra has rank $1$, using that both the Lie algebras $\ngo_X$ and $\ngo_Y$ have rank one, where $\ngo_X$ and $\ngo_Y$ are the subalgebras spanned by the vectors $X_i$ and $Y_j$ respectively. Hence every real semisimple derivation is conjugate to $\lambda D$ for $\lambda \in \mathbb{R}$ and will have an eigenvalue $0$ on the center.

Now we show that the Lie algebras of (iii) are pairwise non-isomorphic and thus give us a one-parameter family of examples.  We denote $\gamma_2(\ngo):=[\ngo,\ngo]$, $\gamma_3(\ngo):=[\ngo,[\ngo,\ngo]]$ and so on.  Note that $\gamma_4(\ngo_t) = \langle X_5, X_6, X_7, Y_5 \rangle$ and thus the centralizer is given by
$$
C(\gamma_4(\ngo_t)) = \langle X_3, X_4, X_5, X_6, X_7, Y_1, Y_2, Y_3, Y_4, Y_5, Z \rangle.
$$
Now define the subspaces
\begin{align*}
U &:= [ C(\gamma_4(\ngo_t)), C(\gamma_4(\ngo_t))] = \langle X_7, Y_3, Y_4, Y_5 \rangle, \\
V &:= C(U) = \langle X_1, \ldots, X_7, Y_3, Y_4, Y_5, Z \rangle.
\end{align*}
Similarly, $\gamma_5(\ngo_t) = \langle X_6, X_7 \rangle$ and
$$W := C(\gamma_5(\ngo_t)) \cap V = \langle X_2, \ldots, X_7, Y_3, Y_4, Y_5, Z \rangle.$$

Let $\varphi: \ngo_s \to \ngo_t$ be an isomorphism and $V, W \subseteq \ngo_s$, $V^\prime, W^\prime \subseteq \ngo_t$ the subspaces as constructed above. These subspaces are characteristic, in the sense that $\varphi(V) = V^\prime$ and $\varphi(W) = W^\prime$. Let $\lambda, \mu, a \in \mathbb{R}$ such that $\varphi(X_1) = \lambda X_1^\prime + a X^\prime_2 + \gamma_2(\ngo_t)$ and $\varphi(X_2) = \mu X_2^\prime + \gamma_2(\ngo_t)$. A computation shows that $\mu = \lambda^2$ and that $$\varphi(X_i) = \lambda^i X_i^\prime + \gamma_{i}(\ngo_t)$$ for all $i \geq 2$. So in particular, we get that $$s \lambda^7 X_7^\prime =  \varphi(s X_7) = \varphi([X_2,X_5]) = [\varphi(X_2),\varphi(X_5)] = \lambda^7 [X_2^\prime,X_5^\prime] = t \lambda^7 X_7^\prime.$$
Because $\lambda \neq 0$ the claim follows.
\end{proof}

In view of the above proposition, besides $\phi_\ngo\ne 0$, we may add to question (Q2) the existence of a non-singular derivation.  The following proposition shows that this does not suffice either.

\begin{proposition}\label{ex8ex7}
There exist nilpotent Lie algebras with $\phi_\ngo\ne 0$ and the following properties:
\begin{itemize}
\item[(i)] $\dim{\ngo}=13$, $\ngo$ is $5$-step nilpotent, $\dim{\zg(\ngo)}=3$, $\rank(\ngo)=1$ and
$$
 D=\Diag(1,2,3,4,5,6,7,-1,-2,-3,-4,-5,1)\in\Der(\ngo), \qquad D|_{\zg(\ngo)}=\Diag(7,-5,1).
$$
\item[(ii)] $\dim{\ngo}=17$, $\ngo$ is $5$-step nilpotent, $\dim{\zg(\ngo)}=4$, $\rank(\ngo)=1$ and
$$
\begin{array}{c}
 D=\Diag(-1,-2,-3,-4,-5,-6,-7,1,2,3,4,5,6,7,-1,2,1)\in\Der(\ngo), \\ \\
 D|_{\zg(\ngo)}=\Diag(-7,7,-1,1), \qquad \tr{D}=2.
\end{array}
$$
\end{itemize}
\end{proposition}

\begin{proof}
Part (i).  Let $\ngo$ be the Lie algebra with nice basis $X_1, \ldots, X_{7}, Y_1, \ldots, Y_5, Z$ and bracket
		\begin{align*}
			[X_1, X_3] &= X_4  &[X_1, X_4] &= X_5 &[X_1, X_5] &= X_6 \\
			[X_1, X_6] &= X_7 &[X_2, X_3] &= X_5 &[X_2, X_4] &= X_6 \\
			[X_3, X_4] &= -X_7 &[X_2, X_5] &= X_7 & & \\
			[Y_1, Y_2] &= Y_3 &[Y_1, Y_3] &= Y_4  &[Y_1, Y_4] &= Y_5\\
			[Y_2, Y_3] &= Y_5 &[X_2, Y_1] &= Z &[X_3, Y_2] &= Z.
		\end{align*}
The center is $\zg(\ngo)=\la X_7,Y_5,Z\ra$. 		Let $\ngo_X$, $\ngo_Y$ and $\ngo_Z$ be the vector spaces spanned by the vectors $X_i$, $Y_j$ and $Z$ respectively. These are all Lie subalgebras of $\ngo$ of rank $1$, which follows from \cite{Mgn} or from an explicit computation.  Consider the invertible derivation $D$ defined as $D(X_i) = i X_i, D(Y_j) = -j Y_j$ and $D(Z) = Z$ and let $D^\prime$ be any real semisimple derivation commuting with $D$. The subalgebra $\ngo_Z$ is invariant under $D^\prime$ since it is the intersection of $[\ngo,\ngo]$ and the eigenspace of $D$ of eigenvalue $1$ and similarly,
$$
D^\prime(\ngo_X \oplus \ngo_Z) \subseteq \ngo_X \oplus \ngo_Z, \qquad
D^\prime(\ngo_Y) \subseteq \ngo_Y.
$$		
%D^\prime(\ngo_X \oplus \ngo_Z) \subseteq \ngo_X \oplus \ngo_Z, \qquad
%D^\prime(\ngo_Y \oplus \ngo_Z) \subseteq \ngo_Y \oplus \ngo_Z,
%$$
%as these subspaces are the sum of the eigenspaces of $D$ with positive eigenvalues (respectively negative eigenvalues) with $\ngo_Z$.
%
Consider now the induced map by $D^\prime$ on $\faktor{\ngo} {\ngo_Y \oplus \ngo_Z } \approx \ngo_X$ and $\faktor{\ngo}{ \ngo_X \oplus \ngo_Z } \approx \ngo_Y$. Since these quotients have rank one, we find that $D^\prime(X_i) \in  \lambda i X_i + \ngo_Z$ and $D^\prime(Y_j) \in  - \mu j Y_j + \ngo_Z$ for some $\lambda, \mu \in \mathbb{R}$ and every $1 \leq i \leq 7, 1 \leq j \leq 5$. Now
\begin{align*}
D^\prime(Z) &= [D^\prime(X_3), Y_2] + [X_3,D^\prime(Y_2)] = 3 \lambda Z  - 2 \mu Z \\	
& = [D^\prime(X_2), Y_1] + [X_2,D^\prime(Y_1)] = 2 \lambda Z - \mu Z,
\end{align*} and hence $\lambda = \mu$. Moreover, $X_1$ and $Z$ are eigenvectors of $D'$ for the eigenvalues $\lambda$, so $D^\prime = \lambda D$ since it is real semisimple.

\vspace{.3cm}
\noindent
Part (ii).  Define the Lie algebra $\ngo$ with nice basis $X_1, \ldots, X_7, Y_1, \ldots, Y_7, Z_1, Z_2, Z_3$ and bracket
								\begin{align*}
				[X_1, X_3] &= X_4 &[X_1, X_4] &= X_5 &[X_1, X_5] &= X_6\\
				[X_1, X_6] &= X_7 &[X_2, X_3] &= X_5 &[X_2, X_4] &= X_6\\
				[X_3, X_4] &= -X_7 &[X_2, X_5] &= X_7 & & \\
				[Y_1, Y_3] &= Y_4 &[Y_1, Y_4] &= Y_5 &[Y_1, Y_5] &= Y_6\\
				[Y_1, Y_6] &= Y_7 &[Y_2, Y_3] &= Y_5 &[Y_2, Y_4] &= Y_6\\
				[Y_3, Y_4] &= -Y_7 &[Y_2, Y_5] &= Y_7 & & \\ 				
				[X_3, Y_2] &= Z_1 &[X_2, Y_1] &= Z_1 & &\\
				[Y_3, X_1] &= Z_2 &[Z_2, X_1] &= Z_3. & &
				\end{align*}
Define the subalgebras $\ngo_X$ and $\ngo_Y$ as in (i), then almost the same computations show that $\ngo$ has rank $1$. Any real semisimple derivation is conjugate to a multiple of the derivation $D$ given by $D(X_i) = i X_i$, $D(Y_j)= -j Y_j$ and hence $D(Z_1) = Z_1, D(Z_2) = -2 Z_2, D(Z_3) = -Z_3$. The center is the vector space spanned by $X_7, Y_7, Z_1$ and $Z_3$ and thus $D$ has trace $0$ when restricted to the center. Note that the $\tr(D) = -2 \neq 0$.
\end{proof}

Note that none of the Lie algebras in Propositions \ref{ex1ex2ex5} and \ref{ex8ex7} can be a RN-nilradical, since every non-trivial real semisimple derivation has either a zero or negative eigenvalue on the center.  It follows from Theorem \ref{dim7} that nilpotent Lie algebras such that any real semisimple derivation has a zero eigenvalue can be RN-nilradicals.  We may ask whether the existence of a negative eigenvalue for every real semisimple derivation could be an obstruction to be a RN-nilradical.  The answer is no, as the following example shows.

\begin{example}\label{ex9}
Consider the Lie algebra $\ngo$ with nice basis $X_1, \ldots, X_{6}, Y_1, \ldots, Y_4$ and bracket
\begin{align*}
[X_1, X_2] &= X_3 &[X_1, X_3] &= X_4 &[X_1, X_4] &= X_5\\
[X_1, X_5] &= X_6 &[X_2, X_3] &= X_6 & & \\
[X_1, Y_1] &= Y_2 &[X_1, Y_2] &= Y_4 & [X_2, Y_1] &= Y_3 &[Y_1, Y_3] &= Y_4.
\end{align*}
So $\dim{\ngo}=10$, $\ngo$ is $5$-step nilpotent with descendent central series dimensions $(10,7,4,2,1)$ and $\zg(\ngo)=\la X_6,Y_4\ra$.  Let $D$ be the derivation which maps $D(X_1) =  X_1, D(X_2) = 3 X_2$ and $D(Y_1) = - Y_1$.  We show that $\ngo$ has rank $1$. Let $D^\prime$ be a real semisimple derivation which commutes with $D$. First assume that $D^\prime$ is diagonal. Write $D^\prime(X_1) = \lambda_1 X_1, D^\prime(X_2) = \lambda_2 X_2$ and $D^\prime(Y_1) = \mu Y_1$ for $\lambda_1, \lambda_2, \mu \in \mathbb{R}$. By applying $D^\prime$ to bracket $4$ and $5$, we find that $\lambda_2 = 3 \lambda_1$. Furthermore, the brackets which result to $Y_4$ show that $\mu = - \lambda_1$. So $D^\prime = \lambda_1 D$ and the claim follows.  Now let $D^\prime$ be a general real semisimple derivation which commutes with $D$. Since the basis is nice, the diagonal part also is a derivation which commutes with $D$ and hence the diagonal is equal to $\lambda_1 D$ for some $\lambda_1 \in \mathbb{R}$. Since $D^\prime$ is real semisimple, this implies that $D^\prime$ is equal to its diagonal.

Finally, we show that $\ngo$ is a RN-nilradical. Write $F_1, F_2, F_3$ for the weights corresponding to the brackets $[X_1,Y_2] = Y_4, [X_2,Y_1] = Y_3, [Y_1,Y_3] = Y_4$. Note that for
$$
M := \frac{1}{6} F_1 + \frac{2}{3} F_2 + \frac{2}{3} F_3 \in \RR_{>0}\CH_{\lb},
$$
it holds that $M(Y_1) = - \frac{4}{3} Y_1$, $M(Y_2) = - \frac{1}{6} Y_2$, $M(Y_3) = 0$, $M(Y_4) = \frac{5}{6} Y_4$ and $MX_i=m_iX_i$ with $m_i\leq 0$ for all $i$. We conclude that $D \in M + \tg^n_{>0} \subseteq  \RR_{>0}\CH_{\lb} + \tg^n_{>0} $ and thus Corollary \ref{cor-nice} implies that $\ngo$ is a RN-nilradical.
\end{example}

We now show that, unexpectedly, a characteristically nilpotent Lie algebra can admit a nice basis.

\begin{example}\label{ex4}
We give two examples of dimension $12$ with a nice basis.  The first example $\ngo_1$ has basis $X_1, X_2, X_3, Y_1, Y_2, Y_3, Z_1, Z_2, Z_3, U_1, U_2, U_3$ and bracket
\begin{align*}
[X_1, X_2] &= Y_1 &[X_2, X_3] &= Y_2  &[X_3, X_1] &= Y_3 \\
[X_1, Y_1] &= Z_1 &[X_2, Y_2] &= Z_2  &[X_3, Y_3] &= Z_3 \\
[X_1, Z_1] &= U_1 &[X_2, Z_2] &= U_2 &[X_3, Z_3] &= U_3 \\
[X_1, Y_3] &= U_3, &[X_2, Y_1] &= U_1 &[X_3, Y_2] &= U_2.
\end{align*}
So $\ngo$ is $4$-step nilpotent with descendent central series dimensions $(12,9,6,3)$ and $\zg(\ngo)=\la U_1,U_2,U_3\ra$.  Let $D$ be any derivation of $\ngo_1$. The diagonal of $D$ is again a derivation and an easy computation shows that this must be $0$.
%In particular, we find that $$D(Z_i) = [D(X_i),Y_i] + [X_i,D(Y_i)] \in \gamma_4(\ngo_1)$$ for all $i \in \{1,2,3\}$. So $D(\gamma_3(\ngo_1)) \subseteq \gamma_4(\ngo_1)$.
Now write $D(X_1) = a X_2 + b X_3 + V$ and $D(Y_2) = c Y_1 + d Y_3 + e Z_1 + W$ where $V$ and $W$ are a linear combination of the other basis vectors. Consider $$0 = D([X_1, Y_2]) = [D(X_1), Y_2] + [X_1,D(Y_2)] = a Z_2 + b U_2 + c Z_1 + d U_3 + e U_1,$$ which implies that $a = b = 0$. A similar computation for $X_2$ and $X_3$ shows that $$D(\ngo_1) \subseteq [\ngo_1,\ngo_1]$$ which implies that $\ngo_1$ is characteristically nilpotent.

For the second example, we consider the Lie algebra $\ngo_2$ with basis $X_1, \ldots, X_5, Y_1, \ldots, Y_5, Z_1, Z_2$ and bracket
\begin{align*}
[X_1, X_2] &= X_3 &[X_1, X_3] &= X_4 &[X_1, X_4] &= X_5 &[X_2, X_3] &= X_5 \\
[Y_1, Y_2] &= Y_3 &[Y_1, Y_3] &= Y_4 &[Y_1, Y_4] &= Y_5 &[Y_2, Y_3] &= Y_5 \\
[X_1, Y_1] &= Z_1 &[X_2, Y_2] &= Z_1 &[X_1, Y_2] &= Z_2 &[X_2, Y_1] &= Z_2.
\end{align*}
Let $D$ be any derivation of $\ngo_2$, then again the diagonal is $0$ by an easy computation. Now write $D(X_1) = a X_2 + b Y_1 + c Y_2 + V$ and $D(X_2) = a^\prime X_1 + b^\prime Y_1 + c^\prime Y_2 + W$ with $V, W \in \gamma_2(\ngo_2)$, then
\begin{align*}
0 = D([X_1,Y_3]) = [D(X_1),Y_3] + [X_1,D(Y_3)] = b Y_4 + c Y_5 + d_1 X_4 + d_2 X_5\\
0 = D([X_2,Y_3]) = [D(X_2),Y_3] + [X_2,D(Y_3)] = b^\prime Y_4 + c^\prime Y_5 + d_3 X_4 + d_4 X_5
\end{align*} for some $d_i \in \mathbb{R}$. Hence $b = c = b^\prime = c^\prime = 0$. Now consider $$D(X_5) = D([X_2,X_3]) = [D(X_2),X_3] + [X_2,D(X_3)] = a^\prime X_4 + d^\prime X_5$$ for some $d^\prime \in \mathbb{R}$ and hence $a^\prime = 0$, since $D(\gamma_4(\ngo_2) \subseteq \gamma_4(\ngo_2)$. A similar computation for $Y_1$ and $Y_2$ shows that $D$ is nilpotent since $D^2(\ngo_2) \subseteq \gamma_2(\ngo_2)$. With some more work, one can show that $D(\ngo_2) \subseteq \ngo_2$, but since we don't need this fact, we don't give the proof here.
\end{example}

\end{document}